\documentclass[11pt]{amsart}
\usepackage{geometry}                
\usepackage{graphicx,comment}
\usepackage{amssymb, amsthm}
\newtheorem{thm}{Theorem}[section]
\newtheorem*{thm*}{Theorem}
\newtheorem{prop}[thm]{Proposition}

\newtheorem{lemma}[thm]{Lemma}
\newtheorem{cor}[thm]{Corollary}
\newtheorem*{cor*}{Corollary}

\theoremstyle{definition}
\newtheorem*{ack}{Acknowledgment}
\newtheorem{defn}[thm]{Definition}
\newtheorem{conj}{Conjecture}
\newtheorem*{conj*}{Conjecture}
\newtheorem{remark}[thm]{Remark}

\newtheorem{question}{Question}
\newtheorem{prob}{Problem}
\newtheorem{claim}{Claim}

\newtheorem*{cor1}{Corollary \ref{cor:decunivth2dpc}}
\newtheorem*{cor2}{Corollary \ref{cor:atomic}}
\newtheorem*{cor3}{Corollary \ref{cor:closedgraphproducts}}
\newtheorem*{cor4}{Corollary \ref{cor:join}}
\newtheorem*{cor5}{Corollary \ref{cor:closedclasses}}
\newtheorem*{thm11}{Theorem \ref{cor:univequivstronemb}}
\newtheorem*{cor7}{Corollary \ref{cor:comt}}
\newtheorem*{cor8}{Corollary \ref{cor:tembuni}}
\newtheorem*{cor9}{Corollary \ref{cor:decemb2dim}}
\newtheorem*{cor10}{Corollary \ref{cor:chartameemb}}
\newtheorem*{thm1}{Theorem \ref{cor:decEGE}}
\newtheorem*{thm2}{Theorem \ref{thm:noncommtree}}

\newcommand{\factor}[2]{{\raise0.7ex\hbox{$#1$} \!\mathord{\left/ {\vphantom {#1 {#2}}}\right.\kern-\nulldelimiterspace}
\!\lower0.7ex\hbox{${#2}$}}}
\newcommand{\BN}{\mathbb{N}}
\newcommand{\BZ}{\mathbb{Z}}

\newcommand{\GG}{\ensuremath{\mathbb{G}}}
\newcommand{\BA}{\ensuremath{\mathbb{A}}}
\newcommand{\HH}{\ensuremath{\mathbb{H}}}

\newcommand{\lra}{\leftrightarrows}

\DeclareMathOperator{\Th}{Th}

\DeclareMathOperator{\St}{Star}
\DeclareMathOperator{\defl}{defl}
\DeclareMathOperator{\infl}{infl}

\title{Embeddability and universal theory of partially commutative groups}
\author[M. Casals-Ruiz]{Montserrat Casals-Ruiz}
\address{Mathematical Insitute, Andrew Wiles Building, Woodstock Rd, University of Oxford, Oxford, OX2 6GG, UK}
\email{montsecasals@gmail.com}

\thanks{The author is supported by the Marie Curie International Incoming Fellowship within the 7th European Community Framework Programme and partly supported by the Spanish Government grant MTM2011-28229-C02-02 and Juan de la Cierva Programme}


\begin{document}

\begin{abstract}

The first part of the paper centers in the study of embeddability between partially commutative groups. In \cite{KK}, for a finite simplicial graph $\Gamma$, the authors introduce an infinite, locally infinite graph $\Gamma^e$, called the extension graph of $\Gamma$. They show that each finite induced subgraph $\Delta$ of $\Gamma^e$ gives rise to an embedding between the partially commutative groups $\GG(\Delta)$ and $\GG(\Gamma)$. Furthermore, it is proven that in many instances the converse also holds. Our first result is the decidability of the Extension Graph Embedding Problem: there is an algorithm that given two finite simplicial graphs $\Delta$ and $\Gamma$ decides whether or not $\Delta$ is an induced subgraph of $\Gamma^e$. As a corollary we obtain the decidability of the Embedding Problem for 2-dimensional partially commutative groups. 

In the second part of the paper, we relate the Embedding Problem between partially commutative groups to the model-theoretic question of classification up to universal equivalence. We use our characterisation to transfer algebraic and algorithmic results on embeddability to model-theoretic ones and obtain some rigidity results on the elementary theory of atomic pc groups as well as to deduce the existence of an algorithm to decide if an arbitrary pc group is universally equivalent to a 2-dimensional one.

\end{abstract}
\maketitle
\section{Introduction}

Partially commutative groups, also known as right-angled Artin groups, are a class of groups widely studied on account of their simple definition, their intrinsically rich structure and their natural appearance in several branches of computer science and mathematics. 

Recall that the {\em partially commutative group {\rm(}or pc group for short{\rm)} with commutation graph } $\Gamma$ is the group $\GG(\Gamma)$ with presentation 
$$
\left< V(\Gamma) \mid [x_i, x_j]=1, \hbox{ for each } (x_i, x_j) \in E(\Gamma) \right>,
$$ 
where $\Gamma$ is a finite (undirected) simple graph, with vertex set $V(\Gamma)$ and edge set $E(\Gamma)$.

Crucial examples, which shape the theory of presentations of groups, arise from study of subgroups of pc groups: notably Bestvina and Brady's example of a group which is homologically finite (of type $FP$) but not geometrically finite (in fact not of type $F_2$); and Mihailova's example of a group with unsolvable subgroup membership problem. More recently results of Wise and others have lead Agol to prove that every hyperbolic $3$-manifold is virtually a subgroup of a pc group and to conclude that the virtually fibred conjecture holds. In the light of such results it is natural to ask which groups arise as subgroups of partially commutative groups, and, in particular, when one partially commutative group embeds in another.

\begin{question} [Problem 1.4 in \cite{CSS}, Question 1 in \cite{KK}] \label{q4}
Does there exist an algorithm that given two simplicial graphs $\Delta$ and $\Gamma$ decides whether or not there exists an embedding from $\GG(\Delta)$ to $\GG(\Gamma)$?
\end{question}

Significant progress towards answering  Question \ref{q4} has been made by Kim and Koberda in \cite{KK} using the notion of the extension graph. The extension graph $\Gamma^e$ of a graph $\Gamma$ is defined as follows: vertices of $\Gamma^e$ are in one-to-one correspondence with conjugates of generators of $\GG(\Gamma)$, that is 
$$
V(\Gamma^e)=\{ g^{-1} x g \in \GG(\Gamma) \mid x\in V(\Gamma), g\in \GG(\Gamma)\}
$$ 
and there is an edge in $\Gamma^e$ when the elements associated to the corresponding vertices commute in the group, i.e. 
$$
E(\Gamma^e)= \{ (u, v) \mid [u,v]=1 \hbox{ in } \GG(\Gamma)\}.
$$

In \cite{KK} it is shown that if $\Delta$ is an induced subgraph of the extension graph $\Gamma^e$, then this graph embedding induces a group embedding from $\GG(\Delta)$ to $\GG(\Gamma)$. We call this type of embeddings \emph{extension graph embeddings}, see Definition \ref{def:tameemb}. The converse is shown to hold under additional conditions on $\Gamma$, namely if $\Gamma$ is triangle-free. In this case, $\GG(\Delta)$ embeds in $\GG(\Gamma)$ if and only if $\Delta$ is an induced subgraph of $\Gamma^e$ (\cite[Theorem 10]{KK}). Recall that $\Delta$ is an induced subgraph of $\Gamma$ if it is isomorphic to a full subgraph of $\Gamma$; and if $\Delta$ is not an induced subgraph of $\Gamma$, then $\Gamma$ is said to be $\Delta$\emph{-free}, see Section \ref{sec:pre}.

These results suggested that the extension graph could be the graph-theoretical tool to determine when a pc group is a subgroup of another one or, in other words, that the following two Embedding Problems are equivalent.

\begin{prob}[Embedding Problem]\label{prob:EP}
Determine when there exists an embedding from a pc group $\GG(\Delta)$ to a pc group $\GG(\Gamma)$.
\end{prob} 

\begin{prob}[Extension Graph Embedding Problem]\label{prob:EGEP}
Determine when there exists an extension graph embedding from a pc group $\GG(\Delta)$ to a pc group $\GG(\Gamma)$.
\end{prob}

However, we show in \cite{CDK} that in general, the Embedding Problems are not equivalent: there are pc groups $\GG(\Delta)$ and $\GG(\Gamma)$ for which there exists an embedding from $\GG(\Delta)$ to $\GG(\Gamma)$ but there are no extension graph embeddings, i.e. $\Delta$ is not an induced subgraph of $\Gamma^e$.

\medskip

The aim of this paper is two-fold: on the one hand, to advance in the study of the Embedding Problems between pc groups; on the other hand, to relate them to the model-theoretic question of classification of pc groups up to universal equivalence and obtain some rigidity results on the elementary theory of atomic pc groups. More concretely, in the first part of the paper, we prove the following:

\begin{thm1}
The Extension Graph Embedding Problem is decidable: there exists an algorithm that given two simplicial graphs $\Delta$ and $\Gamma$ decides whether or not $\Delta$ is an induced subgraph of $\Gamma^e$.
\end{thm1}
Note that if $\Gamma$ is not a clique, then the extension graph is infinite and if in addition, $\Gamma$ is not a disjoint union of cliques, then the extension graph is not locally finite.

\smallskip

As we briefly mentioned, in general the extension graph is not sufficient to characterise when a pc group is a subgroup of another one. However, under additional conditions on the target or the source pc group, it is: if $\Delta$ is a path $P_n$ of length $n$, $\Gamma$ is triangle-free or $\Gamma$ is triangle-built (square and $P_3$-free), then the Embedding Problems are equivalent. For these cases, we deduce from Theorem \ref{cor:decEGE} that the Embedding Problem is decidable.

\begin{cor9}
There is an algorithm that given a triangle-free or triangle-built simplicial graph $\Gamma$ and arbitrary simplicial graph $\Lambda$, decides whether or not there exists an embedding from $\GG(\Lambda)$ to $\GG(\Gamma)$.
\end{cor9}

In the remainder of Section \ref{sec:1}, we investigate conditions on the target and source pc group under which the Embedding Problems are equivalent. Among other results, we prove the following:

\begin{thm2}
If $\Delta$ is the complement of a forest, then the Embedding Problems are equivalent: $\GG(\Delta)$ embeds into $\GG(\Gamma)$ if and only if the graph $\Delta$ is an induced subgraph of the extension graph $\Gamma^e$.
\end{thm2}

Recall that the complement $\overline \Delta$ of a graph $\Delta$ is a graph with set of vertices $V(\Delta)$ whose set of edges is the complement of $E(\Delta)$, that is 
$$
E(\overline \Delta) = (V(\Delta) \times V(\Delta)) \setminus (\{ (v,v), v\in V(\Delta)\} \cup E(\Delta)\}).
$$ 
The class of graphs that are complement of forests is ``generic'' in the sense that every pc group $\GG(\Delta)$ embeds into a pc group whose defining graph is the complement of a forest, see \cite{KK2}. Our result raises the following question: in order to determine whether or not the Embedding Problems are equivalent for a given (source) pc group $\GG(\Delta)$, does it suffice to consider target pc groups whose defining graphs are complement of forests? In other words, suppose that whenever $\GG(\Delta)$ embeds into $\GG(\Gamma)$ for some complement of a forest $\Gamma$, we have that $\Delta < \Gamma^e$. Would this imply that whenever $\GG(\Delta)$ embeds into an arbitrary $\GG(\Gamma)$, then $\Delta< \Gamma^e$?

The analogous Embedding Problem for pc monoids, known as codings of trace monoids, was intensively studied in the mid 80's. In \cite{Kunc}, Kunc proves that the Coding Problem for pc monoids is undecidable, that is there is no algorithm to decide whether or not there exists an embedding between two given trace monoids. However, if one restricts to the case when the defining graph of the source pc monoid is the complement of a forest, then the Embedding Problem is decidable. As a corollary of the above theorems, we deduce the analogous result for pc groups, more precisely, we show:

\begin{cor7} {\rm(}cf. Corollary 11.5 in \cite{Kunc}{\rm)}
There is an algorithm that given the complement of a forest $\Delta$ and an arbitrary simplicial graph $\Gamma$ decides whether or not $\GG(\Delta)$ embeds into $\GG(\Gamma)$.
\end{cor7}

\bigskip

The second part of the paper explores the connection between embeddability and universal equivalence. Recall that the \emph{universal theory} of a group is the set of all universal sentences (in the language of groups) that are true in the group, and that two groups are called \emph{universally equivalent} if their universal theories coincide.

Analysing the free abelian groups case, it is immediate to see that embeddability is not sufficient in order to provide a classification up to universal equivalence: $\BZ^n$ is universally equivalent to $\BZ$, for all $n\ge 1$ and obviously $\BZ^n$ does not embed into $\BZ$, if $n>1$. We show that basically this is the only obstruction and that if one replaces the cyclic groups defined by generators of the pc group by free abelian groups of sufficiently large rank, then embeddability (in this bigger group) characterises universal equivalence. More precisely, we prove:

\begin{thm11}
Let $\Delta$ and $\Gamma$ be two simplicial graphs, $|V(\Delta)|=n$ and $|V(\Gamma)|=m$. Then, $\GG(\Delta)$ is universally equivalent to $\GG(\Gamma)$ if and only if there exist (tame) embeddings $\varphi:\GG(\Delta) \to \GG(\Gamma_n)$ and $\psi: \GG(\Gamma) \to \GG(\Delta_m)$, where $\GG(\Delta_n)$ is the graph product with underlying graph $\Delta$ and vertex groups isomorphic to $\BZ^n$ {\rm(}see Definition {\rm\ref{def:infl})}.
\end{thm11}

\medskip

Using the linearity of pc groups, one can characterise universal equivalence in terms of the residual properties of the group as follows, see \cite{remes}: two partially commutative groups are universally equivalent if and only if they discriminate each other. Recall that a group $G$ is said to be \emph{discriminated} by $H$ if for any finite set $S$ of elements of $G$ there exists a homomorphism from $G$ to $H$ that is injective on $S$. Finitely generated groups universally equivalent to (non-abelian) free groups are also known as limit groups. In this terminology, our result gives a characterisation of when a pc group $\GG(\Gamma)$ is a limit group over a pc group $\GG(\Delta)$.

\medskip

The fact that when studying universal equivalence one can replace the cyclic groups defined by generators of the pc groups $\GG(\Delta)$ by free abelian groups of sufficiently large rank allows us to consider special type of embeddings from $\GG(\Gamma)$ to $\GG(\Delta_m)$, called tame, see Definition \ref{def:tameemb}. Roughly speaking, for tame embeddings there are no repetitions of letters in the images of the generators of $\GG(\Delta)$. In this setting, we show that given two arbitrary pc groups $\GG(\Delta)$ and $\GG(\Gamma)$, there exist finitely many elements $w_1, \dots, w_k \in \GG(\Delta)$ and $n$-tuples $\overline{g_1}, \dots, \overline{g_{nm}}$ from $\GG(\Gamma)$ which completely determine the existence of a tame embedding from $\GG(\Delta)$ to $\GG(\Gamma)$, $n=|V(\Delta)|$, $m=|V(\Gamma)|$. More precisely, we prove:

\begin{cor10}
Let $\Delta$ and $\Gamma$ be simplicial graphs, $|V(\Delta)|=n$, $|V(\Gamma)|=m$. There exist words $w_1, \dots, w_k \in \GG(\Delta)$ and $\overline{g_1}, \dots, \overline{g_{nm}} \in \GG(\Gamma)^n$ such that the following are equivalent:

\begin{itemize}
\item there exists a tame embedding from $\GG(\Delta)$ to $\GG(\Gamma)$;
\item at least one of the maps $\varphi_i$ that sends the tuple of generators $V(\Delta)$ to the tuple $\overline{g_i}$ is injective in the set of words $w_1, \dots, w_k$, $i=1, \dots, nm$, or in other words, there exists $i\in \{1, \dots, nm\}$ such that $w_j=1$ in $\GG(\Delta)$ if and only if $\varphi_i(w_j)=1$ in $\GG(\Gamma)$ for all $j=1, \dots, k$.
\end{itemize}
\end{cor10}

Furthermore, the number $k$ and the tuples $\overline g_i$ can be found algorithmically. Therefore, if one could effectively bound the length of the words $w_1, \dots, w_k$ in Corollary \ref{cor:chartameemb}, then one would deduce the decidability of the Tame Embedding Problem. In other words, one would obtain an algorithm that given simplicial graphs $\Delta$ and $\Gamma$ decides whether or not there is a tame embedding from $\GG(\Delta)$ to $\GG(\Gamma)$. We conjecture that the Tame Embedding Problem is decidable and so, if this is the case, we would conclude that one can algorithmically decide whether or not two pc groups are universally equivalent.

\begin{cor8}
If the Tame Embedding Problem is decidable, then there is an algorithm that given two simplicial graphs $\Delta$ and $\Gamma$ decides whether or not $\GG(\Delta)$ and $\GG(\Gamma)$ are universally equivalent.
\end{cor8}

Although, in general, the Tame Embedding Problem is open, we proved that it is decidable for certain classes of pc groups and in such cases the above corollary applies. In particular, we deduce that there is an algorithm to decide when an arbitrary pc group is universally equivalent to a $2$-dimensional one and, more generally, to a  pc group whose deflation graph is triangle-free (see Definition \ref{def:defl}).

\begin{cor1}
There exists an algorithm that given a triangle-free simplicial graph $\Delta$ and a simplicial graph $\Gamma$ decides whether or not $\GG(\Delta)$ and $\GG(\Gamma)$ are universally equivalent.
\end{cor1}

Our characterisation of the universal equivalence in algebraic terms allows us to transfer algebraic results on embeddability to model-theoretic ones and prove the following corollaries.

\begin{cor3}
Universal equivalence in the class of pc groups is preserved under free products, direct products and more generally graph products:
$$
\GG(\Delta_i) \equiv_\forall \GG(\Gamma_i) \ i=1, \dots, |V(\Lambda)|  \Rightarrow \mathcal G(\Lambda, S(\Delta))  \equiv_\forall \mathcal G(\Lambda, S(\Gamma)),
$$
where $\Lambda$ is a simplicial graph, $S(\Delta)=\{\GG(\Delta_1), \dots, \GG(\Delta_{|V(\Lambda)|})\}$, $S(\Gamma)=\{\GG(\Gamma_1), \dots, \GG(\Gamma_{|V(\Lambda)|})\}$ and $\mathcal G(\Lambda,S )$ denotes the graph product with underlying graph $\Lambda$ and vertex groups $S$.
\end{cor3}

\medskip

\begin{cor4}
If $\Gamma$ is not a join and has more than one vertex, then
$$
\Th_\forall(\GG(\Gamma)) = \Th_\forall ( \mathop{\ast}\limits_{i=1}^{n} \GG(\Gamma)).
$$
\end{cor4}

\medskip

\begin{cor5}
The following classes of partially commutative groups are closed under universal equivalence, that is if $\GG(\Gamma)$ belongs to a class $\mathcal A$ described below and $\GG(\Delta)$ is universally equivalent to $\GG(\Gamma)$, then $\GG(\Delta)$ also belongs to $\mathcal A$:
\begin{itemize}
\item the class of pc groups whose deflation graph is a tree {\rm(}see Definition {\rm\ref{def:defl})};
\item the class of pc groups whose deflation graph is triangle-free;
\item the class of weakly chordal pc groups {\rm(}i.e. graphs that do not contain triangles and paths with more than $3$ edges{\rm)}.
\end{itemize}

\end{cor5}

\medskip

In a different vein, we can deduce rigidity results for the elementary theory of atomic pc groups. In \cite{BKS}, Bestvina, Kleiner and Sageev introduce the class of atomic graphs, that is connected graphs with no valence 1 vertices, no cycles of length less than 5 and no separating closed stars of vertices, and establish quasi-isometric rigidity for pc groups defined by this class of graphs: two pc groups defined by atomic graphs are quasi-isometric if and only if they are isomorphic. We show that a similar type of rigidity holds for the elementary theory of pc groups defined by atomic graphs.

\begin{cor2}\
Let $A_1$ and $A_2$ be two atomic graphs. Then the following are equivalent:
\begin{itemize}
\item $\GG(A_1)$ and $\GG(A_2)$ are elementarily equivalent;
\item $\GG(A_1)$ and $\GG(A_2)$ are universally equivalent;
\item $\GG(A_1)$ and $\GG(A_2)$ are isomorphic.
\end{itemize}
\end{cor2}

\begin{ack}
Part of this work was written while the author was visiting the university of the Basque Country. The support of the ERC grant PCG-336983 is gratefully acknowledged. 

The author is grateful to Ilya Kazachkov, Martin Bridson and the referees for their constructive comments.
\end{ack}

\section{Basics on partially commutative groups}\label{sec:pre}

We begin by recalling notions of the theory of partially commutative groups and graph theory.

Throughout this paper, by a graph $\Gamma$ we mean a simplicial, undirected graph. We say that $\Lambda$ is an \emph{induced subgraph} of $\Gamma$ if there is an injective map $i: V(\Lambda)\to V(\Gamma)$ such that $(v,v') \in E(\Lambda)$ if and only if $(i(v),i(v')) \in E(\Gamma)$, for all $v,v'\in V(\Lambda)$. Given a subset $S$ of $V(\Gamma)$, the graph \emph{spanned} by $S$ in $\Gamma$ is the induced subgraph with vertex set $S$. The \emph{complement} of the graph $\Gamma$, denoted by $\overline \Gamma$, is the graph with set of vertices $V(\Gamma)$ and whose set of edges
$$
E(\overline \Gamma) = (V(\Gamma) \times V(\Gamma)) \setminus (\{ (v,v), v\in V(\Gamma)\} \cup E(\Gamma)\})
$$ 
is the complement of $E(\Gamma)$, that is $(v,v') \in E(\overline \Gamma)$ if and only if $(v,v') \notin E(\Gamma)$, for all $v,v'\in V(\overline \Gamma)$, $v\ne v'$. A graph $\Gamma$ is called a \emph{clique} if for all $v,v'\in V(\Gamma)$, $v \ne v'$, we have that $(v,v') \in E(\Gamma)$. A graph $\Gamma$ is called a \emph{join} if $V(\Gamma)$ is the disjoint union of two sets, say $V$ and $V'$, and $(v,v') \in E(\Gamma)$, for all $v\in V, v' \in V'$.

\medskip

Let $A = V(\Gamma) = \{a_1, \dots , a_n\}$ be the set of vertices of $\Gamma$, let $F(A)$ be the free group on the alphabet $A$ and $\GG(\Gamma)$ the pc group with finite defining graph $\Gamma$. Abusing the terminology, we often identify vertices of $\Gamma$ with letters of the alphabet $A$, and elements of $\GG(\Gamma)$ with words in the alphabet $A$. We say that a letter occurs in a word if either the letter or its formal inverse appears in the word. We reserve the term \emph{occurrence} to denote a specific occurrence of a letter or of the formal inverse of a letter in a word. More precisely, an occurrence is a pair (letter or its inverse, its placeholder in the word).

For a  word $w\in \GG(\Gamma)$ denote by $\alpha(w)$ the set of letters occurring in a geodesic representing $w$. We define $\BA(w)$ to be the subgroup of $\GG(\Gamma)$ generated by all letters that do not occur in a geodesic representing $w$ and commute with $w$. Both $\alpha(w)$ and the subgroup $\BA(w)$ are well-defined (independent of the choice of a geodesic). If $w' \in \BA(w)$, we say that $w$ and $w'$ \emph{disjointly commute} and write $w \lra w'$.

Let $g,h \in \GG$. We say that $h$ \emph{left-divides} (\emph{right-divides}) $g$ if there exist geodesic words $w$ and $v$ representing $g$ and $h$ respectively such that $w$ has as initial (final) subword $v$. In this paper, we assume that the reader is familiar with divisibility in pc groups and refer to \cite{EKR} and \cite{CK1} for details and further bibliography on the subject.

For a partially commutative group $\GG=\GG(\Gamma)$, we call the complement of $\Gamma$ the \emph{non-commutation graph} of $\GG$. The graph $\overline \Gamma$ is a union of its connected components, namely $I_1, \ldots , I_k$. Then
\begin{equation}\label{eq:decomp}
\GG= \GG(I_1) \times \cdots \times \GG(I_k).
\end{equation}

Let $w' \in \GG$, $w'=g^{-1} w g$ and $w$ be cyclically reduced. Recall that $w$ is cyclically reduced if the length of $ww$ is twice the length of $w$. Consider the induced subgraph of $\overline \Gamma$ spanned by the set of vertices $\alpha(w)$ and denote it by $\Delta(\alpha(w))$. We will call $w'$ a \emph{block} if the graph $\Delta(\alpha(w))$ is connected. If $\Delta(\alpha(w))$ is not connected, then $w$ can be represented as the product of commuting words

\begin{equation} \label{eq:bl}
w= w_{j_1} \cdot w_{j_2} \cdots w_{j_t};\ j_1, \dots, j_t \in J,
\end{equation}
where $|J|$ is the number of connected components of $\Delta(\alpha(w))$ and $w_{j_i}$ are words in letters from the $j_i$-th connected component. Clearly, the words $\{w_{j_1}, \dots, w_{j_t}\}$ pairwise commute. Each word $w_{j_i}$, $i \in {1, \dots,t}$ is a block and so we refer to presentation (\ref{eq:bl}) as the \emph{block decomposition} of $w$. Note that, in \cite{KK}, the block decomposition is called pure factor decomposition.

An element $w\in \GG$ is called a \emph{least root} (or simply, \emph{root element}) of $v\in \GG$ if there exists $0\ne m\in \BN$ such that $v=w^m$ and there does not exists $w'\in \GG$ and $m'\in \BN$ such that $w={w'}^{m'}$. In this case we write $w=\sqrt{v}$. By \cite{DK}, partially commutative groups have least roots, that is the root element of $v$ exists and is unique.

The following theorem reduces the conjugacy problem for arbitrary elements of a partially commutative group to the one for block elements.

\begin{thm*}[Proposition 5.7 of \cite{EKR}] \label{prop:57}\
Let $w= w_{1} \cdot w_{2} \cdots w_{t}$ and $v= v_{1} \cdot v_{2} \cdots v_{s}$ be cyclically reduced elements decomposed into the product of blocks. Then $v$ and $w$ are conjugate if and only if $s= t$ and, upto index re-enumeration,  $w_{i}$ is  conjugate to $v_{i},\, i= 1, \dots ,t$.
\end{thm*}

\begin{cor*} \label{cor:prop:57}
Let $w= w_1^{r_1} \cdot w_2^{r_2} \cdots w_t^{r_t}$ and $v= v_1^{l_1} \cdot v_2^{l_2} \cdots v_s^{l_s}$ be cyclically reduced elements decomposed into the product of blocks, where $w_i$ and $v_j$ are root elements, $l_i,r_j \in \BN$, $i=1,\dots, t$, $j=1,\dots, s$. Then $w$ and $v$ are conjugate if and only if $s=t$ and, upto index re-enumeration, $r_i=l_i$ and $w_i$ is conjugate to $v_i$, $i= 1, \dots, t$.
\end{cor*}

The next result describes centralisers of elements in partially commutative groups. 

\begin{thm}[Centraliser Theorem, Theorem 3.10, \cite{DK}] \label{thm:centr} \
Let $w\in \GG$ be a cyclically reduced word, $w=v_1\dots v_k$ be its block decomposition. Then, the centraliser of $w$ is the following subgroup of $\GG$:
\begin{equation} \notag
C(w)=\langle \sqrt{v_1}\rangle \times \cdots \times \langle \sqrt{v_k} \rangle\times \BA(w).
\end{equation}
\end{thm}

\begin{cor*} \label{cor:centr}
For any $w\in \GG$ we have that $C(w)=C(\sqrt{w})$.
\end{cor*}

\medskip

We finish this section by recalling some results on embeddability. In \cite{KK}, Kim and Koberda conducted a systematic study of the Embedding Problem and proved, among other things, that if there exists an embedding from a pc group to another one, then there exists a ``nice" one. This means that there exists an embedding such that the image of each generator $x_i$ of the source pc group $\GG(\Delta)$ is a product of pairwise commuting conjugates of powers of generators of the target pc group $\GG(\Gamma)$, i.e. for all $x_i \in V(\Delta)$, 
$$
\varphi(x_i) = y_{i1} \cdots y_{ir_i},
$$
where $y_{ij} = g_{ij}^{-1} x_{ij}'^{n_{ij}} g_{ij}$, $x_{ij}' \in V(\Gamma)$, $n_{ij}\in \BZ$, $g_{ij} \in \GG(\Gamma)$ and $[y_{is}, y_{it}]=1$ for all $1\le s<t\le r_i$.

\begin{defn}\label{def:tameemb}

In the above notation, if all conjugates of generators that appear in the images of the generators of $\GG(\Delta)$ are different (independently of the power in which they appear), i.e. if $g_{ij}^{-1} x_{ij}' g_{ij} \ne g_{kl}^{-1} x_{kl}'g_{kl}$ for all $(i,j) \ne (k,l)$, we say that the embedding is \emph{tame}. 

If in addition, the image is just one conjugate of a generator, i.e. $r_i=1$, $i=1, \dots, |V(\Delta)|$, we call the embedding an \emph{extension graph embedding}. 
\end{defn}

In the latter case, we have that the graph $\Delta$ is an induced subgraph of the \emph{extension graph} $\Gamma^e$ of $\Gamma$  and hence the name. 

\medskip

In the introduction we have formulated the Embedding and Extension Graph Embedding Problems, see Problems \ref{prob:EP} and \ref{prob:EGEP}. Naturally, the analogous problem arises for the intermediate class of tame embeddings.

\begin{prob}[Tame Embedding Problem]\label{prob:TEP}
Determine when there exists a tame embedding from a pc group $\GG(\Delta)$ to a pc group $\GG(\Gamma)$.
\end{prob}

We further refer to Problems \ref{prob:EP}, \ref{prob:EGEP} and \ref{prob:TEP} as the Embedding Problems.

\section{Embeddability between pc groups}\label{sec:1}

The first goal of this section is to prove that the Extension Graph Embedding Problem is decidable, or in other words, there is an algorithm that given two simplicial graphs $\Delta$ and $\Gamma$ decides whether or not $\Delta$ is an induced subgraph of the extension graph $\Gamma^e$.

\smallskip

In the second part of this section, we study conditions on the graphs $\Delta$ and $\Gamma$ which imply that the Embedding Problems are equivalent. More precisely, we show that if the graph that defines the source pc group is the complement of a forest or if the deflation graph of the target pc group is triangle-free, then the Embedding Problems are equivalent. 

\smallskip
Finally, we combine these results to obtain the decidability of the Embedding Problem for some families of pc groups. For instance, we deduce that the Embedding Problem for $2$-dimensional pc groups is decidable: there is an algorithm that given a pc group $\GG(\Delta)$ and a $2$-dimensional pc group $\GG(\Gamma)$ decides whether or not $\GG(\Delta) < \GG(\Gamma)$. Similarly, we show that the Embedding Problem is decidable if the source pc group is defined by the complement of a forest. As we mentioned in the introduction, Kunc proved an analogous result for pc monoids.

\subsection{Decidability of the Extension Graph  Embedding Problem}

\
\smallskip

In this section we show that there is an algorithm that given simplicial graphs $\Delta$ and $\Gamma$ decides whether or not the graph $\Delta$ is an induced graph of the extension graph $\Gamma^e$. Note that if the graph $\Gamma$ is not a clique, or equivalently if $\GG(\Gamma)$ is not abelian, then the extension graph $\Gamma^e$ is infinite. Furthermore, if $\Gamma$ is not a disjoint union of cliques, then the extension graph is not locally finite.

The strategy of proof is very natural although its formalisation is a bit technical. If $\Delta$ is an induced subgraph of $\Gamma^e$, then, by definition, the image of every vertex $x_i \in V(\Delta)$ is a conjugate of a generator $y_i^{g_i}$, $y_i\in V(\Gamma^e), g_i \in \GG(\Gamma)$. The key idea is that if there is an extension graph embedding, then there is one with the property that the conjugators have length bounded above by a constant $R=R(\Delta, \Gamma)$. More precisely, we prove the following:

\begin{thm}\label{thm:boundembd}
Let $\Delta$ and $\Gamma$ be simplical graphs, $V(\Delta)=\{x_1, \dots, x_n\}$, $V(\Gamma)=\{y_1, \dots, y_M\}$ and let $K$ be the number of connected components of $\Delta$. Define $R$ to be $4 K n^2 M^{K+1}$. If $\Delta$ is an induced subgraph of $\Gamma^e$, then $\Delta$ is an induced subgraph of $B(\Gamma^e, R)$, where $B(\Gamma^e, R)$ is the induced subgraph of $\Gamma^e$ spanned by the finite set of vertices $\{ y_i^{v} \mid |v| \le R\}$.
\end{thm}

Note that the number $R$ is not optimal. Indeed, if $V(\Delta) \subset V(\Gamma)$, then $\Delta$ is an induced subgraph of $B(\Gamma^e, 1)$.

\medskip

Before turning our attention to the proof of Theorem \ref{thm:boundembd}, we present some technical lemmas.

\begin{lemma}\label{lem:1}
Let $a\in V(\Delta)$ and $u,v \in \GG(\Delta)$. If $a^u, a^v$ are geodesic and $[a^u,a^v]=1$, then $a^u=a^v$.
\end{lemma}

\begin{proof}
Assume that $[a^{u}, a^{v}]=1$ and so $[a, a^{u v^{-1}}]=1$. Let $u = u' d$ and $v=v'd$, where $d$ is the greatest common right-divisor of $u$ and $v$. Then we have $a^{ uv^{-1}}=a^{u'{v'}^{-1}}$. Note that $a^{u'{v'}^{-1}}$ may be non-geodesic. Since $[a, a^{u{v'}^{-1}}]=1$, it follows from the description of centralisers, see Theorem \ref{thm:centr}, that $[a, u']=1$ and $[a,v']=1$. Since, by assumption, $a^{u} = a^{u'd}$ and $a^{v}=a^{v'd}$ are geodesics, we have that $u'= v'=1$ and so $u=v$.
\end{proof}

\begin{lemma}\label{lem:2}
Let $a_1,a_2 \in V(\Gamma)$ and $r\in \GG(\Gamma)\setminus \{1\}$. Suppose that no letter in the centraliser $C(a_1)$ of $a_1$ left-divides $r$ and no letter in the centraliser $C(a_2)$ of $a_2$ right-divides $r$. Then there exist pairwise distinct letters $c_1, \dots, c_m \in V(\Gamma)$ such that 
\begin{itemize}
\item $r= g_m c_m \cdots g_1 c_1 g_0$, for some $g_i \in \GG(\Gamma)$,
\item $c_1$ right-divides $r$, 
\item $c_m$ left-divides $r$ and 
\item the tuple $(a_1, c_m, \dots, c_1, a_2)$ defines a path in the non-commutation graph $\overline \Gamma$.
\end{itemize}
\end{lemma}

\begin{proof}
We prove the statement by induction on the length of $r$. If $|r|=1$, that is $r=c$, $c\in V(\Gamma)$, then since $c$ right and left-divides $r$, by assumption it does not commute with either $a_1$ or $a_2$. Hence $c$ satisfies the statement of the lemma and the base of induction is proven. 

We now prove the step of induction. Let $c_1$ be a right-divisor of $r$. If the occurrence $c_1$ left-divides $r$, then $c_1$ satisfies the statement of the lemma. If $c_1$ does not left-divide $r$, then there exists a letter $c_{2}$ in $r$ that does not commute with $c_1$ and $r= g_2' c_{2} g_1 c_1 g_0$. By induction on the length of $g_2' c_{2}$, there exist pairwise letters $c_m, \dots, c_2$ such that $r= g_m c_m \cdots g_2 c_2 g_1 c_1 g_0$, $c_1$ right-divides $r$, $c_m$ left-divides $r$, and the tuple $(a_1, c_m, \dots, c_1, a_2)$ defines a path in the non-commutation graph $\overline \Gamma$. If $c_1 \ne c_i$, $i=2, \dots, m$, then $c_1, \dots, c_m$ satisfy the statement of the lemma. Assume that $c_1 = c_i$, for some $i\in \{2, \dots, m\}$. Let $i_0 = \max \{ i \in \{2, \dots, m\} \mid c_1 = c_i\}$. Then $c_m, \dots, c_{i+1}, c_1$ satisfy the statement of the lemma.
\end{proof}

\begin{lemma}\label{claim1}
Let $a, b,c \in V(\Gamma)$ and $u,w \in \GG(\Gamma)$. Assume that $c$ right-divides $u$ but does not right-divide $w$. If $a^u$ and $b^w$ are geodesic and commute with each other, then $c$ disjointly commutes with $b^w$.

\end{lemma}

\begin{proof}

Since $c$ right-divides $u$ we have that $u=u_1 c$ for some $u_1 \in \GG(\Gamma)$. Let $g= w ( c^{-1} u_1^{-1} a u_1 c) w^{-1}$. From cancellation theory in partially commutative groups, see Section 3 and Lemma 4.9 in \cite{CK1}, and since $c$ does not right-divide $w$, it follows that $w^{-1} = d_1 d_2 d_3$ and $u_1= u_1'' d_1^{-1}$, $d_1 \lra c$, $d_2 \lra c^{-1} u_1''^{-1} a u_1'' c$ and $d_3^{-1} c^{-1} u_1''^{-1} a u_1''cd_3$ is a geodesic word representing $g$.

Furthermore, $[a^{u}, b^{w}]=1$ implies that $[g, b]=1$ and so from the description of centralisers, see Theorem \ref{thm:centr}, we have that $[d_3, b]=1$. It follows that $d_3=1$ since otherwise $d_3$ right-divides $w^{-1}$ and commutes with $b$ - a contradiction with the fact that $b^{w}$ is geodesic. Since $d_3$ is trivial $w^{-1}=d_1d_2$ disjointly commutes with $c$. Furthermore, since $c$ is a letter in the geodesic $g$, it commutes with $b$ and so we conclude that $c$ disjointly commutes with $b^{w}$. 
\end{proof}

\medskip

We now turn to the proof of Theorem \ref{thm:boundembd}.

\begin{proof}[Proof of Theorem \ref{thm:boundembd}]
Throughout this proof, we denote by $a^v$ the element $v a v^{-1}$. Abusing the terminology, we call the subgraph $B(\Gamma^e, R)$ of $\Gamma^e$ induced by conjugates of generators so that the conjugator has length bounded by $R$ the \emph{ball} of radius $R$ in $\Gamma^e$, that is $B(\Gamma^e, R)$ is the induced subgraph of $\Gamma^e$ with vertex set $V(B(\Gamma^e, R))=\{ y_i^{v} \mid |v| \le R\}$. Assume that $\Delta$ is an induced subgraph of $\Gamma^e$ and let $f$ be a graph embedding from $\Delta$ to $\Gamma^e$.

\smallskip

We first prove the statement for connected graphs. The proof of the lemma for general graphs is by induction on the number of connected components.

\smallskip

Assume that the graph $\Delta$ is connected. Let $\{a_1^{w_1}, \dots, a_n^{w_n}\}$ be the images of the vertices of $\Delta$ in $\Gamma^e$, that is $f(x_i)=a_i^{w_i}$ and $a_i \in \{y_1, \dots, y_M\}$. Without loss of generality, we assume that $a_i^{w_i}$ is geodesic.

Our goal is to construct an embedding $f'$ from $f$ such that the images $f'(x_i)= a_i^{v_i}$ have the property that the length of the conjugators $v_i$ is bounded by $R<n^2M$, for each $i=1, \dots,n$ and where $a_i \in \{y_1, \dots, y_M\}$.

Since our proof is technical, we first give an outline. We consider a graph $\Delta'$ with $n$ vertices, $n= |V(\Delta)|$. To each vertex in $\Delta'$, we associate the generator $a_i$ of $\Gamma$ defined by the image $f(x_i)=a_i^{w_i}$, $i=1, \dots, n$. Note that different vertices of $\Delta'$ can have associated the same generator of $\Gamma$. By definition, two vertices of  $\Delta'$ are joined by an edge if and only if  the corresponding generators commute, that is $[a_i,a_j]=1$. In particular, if $a_i=a_j$ for some $i\ne j$, then we add an edge $(a_i,a_j)$. We call $\Delta'$ the \emph{extended} commutation graph of the set $\{a_1, \dots, a_n\}$.

Using the natural correspondence $x_i \to a_i$ between vertices of $\Delta$ and $\Delta'$, it follows from the definition of $\Delta'$ and the description of centralisers in pc groups, see Theorem \ref{thm:centr}, that $E(\Delta) \subset E(\Delta')$. Our proof proceeds  by induction on the number of edges in $E(\Delta') \setminus E(\Delta)$. At every step we fix an edge $e\in E(\Delta')\setminus E(\Delta)$. The idea is to use the words $w_i$ to construct the words $v_i'$ in a ``minimal way'' such that, on the one hand, we control their length and, on the other one, after conjugating the vertices of $\Delta'$ by $v_i'$, we obtain a new graph $\Delta^{2}$ such that $E(\Delta) \subset E(\Delta^2) \subset E(\Delta')\setminus \{e\}$. After finitely many steps we obtain a graph $\Delta^k$ which is equal to $\Delta$ and is a subgraph of $B(\Gamma^e, R)$, where $R$ is explicitly defined in terms of $\Delta$ and $\Gamma$.

\medskip

Let us now turn our attention to the proof. Abusing the notation, we often identify vertices $x_i \in \Delta$ with their images $a_i^{w_i}$ in $\Gamma^e$.

Let $E(\Delta') \setminus E(\Delta) =\{e_1, \dots, e_k\}$. Our goal is to prove by induction the following statement $\Phi(k,k')$:

For each $k\ge 0$ and for each $1 \le k'\le k$, there exists an integer $0 \le m \le k'|V(\Gamma)|$ and distinct letters $c_1, \dots, c_m$ such that for each $s= 0, \dots, m$ and for each $i=1, \dots, |V(\Delta)|$, there is $0\le q=q(i,s) \le m$ and $w_{i,s}'=c_q \cdots c_1$ (where, by convention, $c_0=1$), satisfying that we can mark letters $c_q, \cdots, c_1$ appearing in $w_i$ in this order, and that:

\begin{itemize}
\item {(IH1)}: if $[a_i^{w_i}, a_j^{w_j}]=1$, then $[a_i^{w_i'}, a_j^{w_j'}]=1$, $i,j \in \{1, \dots, |V(\Delta)|\}$;
\item {(IH2)}: the marked letters $c_p$ cancel in the product $w_iw_j^{-1}$ if and only if $p\le \min\{q(i,s),q(j,s)\}$;
\item {(IH3)}: For each $l=1, \dots, k'$, if we write $e_l=(a_{i_l}, a_{j_l})$, then  $[a_{i_l}^{w_{i_l,m}'}, a_{j_l}^{w_{j_l',m}}] \ne 1$;
\end{itemize}

We begin with the following observation: if $f(x_i)=a_i^{w_i}$, $f(x_j)=a_j^{w_j}$ and $a_i=a_j$, $i\ne j$, then $(x_i,x_j) \notin E(\Delta)$. Indeed, if $(x_i,x_j) \in E(\Delta)$, since $f$ is an embedding, we have that $(f(x_i),f(x_j)) \in E(\Gamma^e)$ and so by definition of the extension graph, $[a_i^{w_i}, a_j^{w_j}] = [a_i^{w_i}, a_i^{w_j}]=1$. Since we assume $a_i^{w_i}$ and $a_i^{w_j}$ to be geodesic, it follows from Lemma \ref{lem:1} that $w_i=w_j$ and so $f(x_i)=a_i^{w_i} = a_j^{w_j} =f(x_j)$ for $i\ne j$, contradicting the fact that $f$ is injective.

\smallskip

Base of induction $\Phi(0, \emptyset)$. In this case, since $k=0$ we have that $E(\Delta) = E(\Delta')$. From the above observation, it follows that the generators associated to vertices in $\Delta'$ are all pairwise distinct and so $\Delta'$ is a subgraph of $\Gamma^e$. Indeed, on the one hand, if $a_i=a_j$ for some $i \ne j$, it follows from the definition of the extended commutation graph $\Delta'$ that $(a_i,a_j) \in \Delta'$. On the other hand, if $(x_i,x_j) \in E(\Delta)$, then $(f(x_i), f(x_j)) = (a_i^{w_i}, a_j^{w_j}) \in E(\Gamma^e)$ and by the above observation $a_i \ne a_j$. Hence if $\Delta' = \Delta$, then $a_i \ne a_j$ for all $i\ne j$. In this case it suffices to take $m=s=0$ and so $q(i,0)=0$ and $w_{i,0}'=1$, for each $i=1, \dots, |V(\Delta)|$, and the induction hypothesis are satisfied.

\medskip

Base of induction for $k'$ at step $k$. In other words, we next prove $\Phi(k,1)$ for a fixed $k$.

Let $e\in E(\Delta') \setminus E(\Delta)$. Without loss of generality, set $e=(a_1, a_2)$. Note that by definition $[a_1, a_2]=1$.

Since $e \notin E(\Delta)$, it follows that $[a_1^{w_1}, a_2^{w_2}] \ne 1$ and so $[a_1^{w_1w_2^{-1}}, a_2]\ne 1$. Since $[a_1, a_2]=1$, it follows that the double-coset representative $C(a_1) w_1w_2^{-1} C(a_2)$ of $w_1w_2^{-1}$ is non-trivial, or equivalently, if  $w_1w_2^{-1}=D_1 r D_2$ where $D_i$ is the maximal divisor of $w_1w_2^{-1}$ from $C(a_i)$, then $r$ is non-trivial. Note that if a letter $c$  right-divides $r$, then it does not commmute with $a_2$. Otherwise if $r=r'c$, $c \in C(a_2)$, then we have that $w_1w_2^{-1}=D_1 r' (c_m D_2)$ and $c_m D_2 \in C(a_2)$ contradicting the choice of $D_2$. Similarly, any letter that left-divides $r$ does not commute with $a_1$. Then by Lemma \ref{lem:2}, there exist pairwise distinct letters $c_1, \dots, c_m$ such that 
\begin{itemize}
\item $r= g_m c_m \cdots g_1 c_1 g_0$, 
\item $c_1$ right-divides $r$, 
\item $c_m$ left-divides $r$ and
\item the tuple $(a_1, c_m, \dots, c_1, a_2)$ defines a path in the non-commutation graph $\overline \Gamma$.
\end{itemize}

Among such paths, we choose a shortest one and mark the right-most occurrences of the letters $c_i$ in $r$, $i=1, \dots, m$. Since there are no repetitions of letters $c_i$ in the path, it follows that $1 \le m \le |V(\Gamma)|$.

We prove the statement $\Phi(k,1)$ by induction on $1 \le m \le |V(\Delta)|$.

Base of induction $m=1$. In this case we have that $s=1$. Then by definition $a_1 c_1 a_2$ defines a path in the non-commutation graph $\overline \Delta$. Let $u_i a_i u_i^{-1}$ be a geodesic word corresponding to the element $a_i^{w_i w_2^{-1} D_2^{-1}}$, that is $a_i^{w_i w_2^{-1} D_2^{-1}}=a_i^{u_i}$ and $u_i a_i u_i^{-1}$ is a geodesic word, $i=1, \dots, n$. In particular, $u_2=1$ and $u_1=r$, $c_1$ right-divides $u_1$ and $u_1=u_1' c_1$.

If $[a_1^{u_1}, a_i^{u_i}]=1$, then it follows from Lemma \ref{claim1} that either $c_1$ right-divides $u_i$ or $c_1$  disjointly commutes with $a_i^{u_i}$.

Let $S_{c_1}$ be the set of vertices $a_i^{w_i}$ of $\Delta$ such that $c_1$ right-divides $u_i$. For each vertex $a_i^{w_i}$ in $S_{c_1}$, we mark the occurrence $c_1$ in $u_i$ that right-divides $u_i$.

\medskip

\begin{claim} \label{claim2}
Any path in $\Delta$ connecting a vertex in $S_{c_1}$ to the vertex $a_2^{w_2}$ contains a vertex $v$ such that $v \notin S_{c_1}$ and $v\ne a_2^{w_2}$.
\end{claim}

Let $v_0 \in S_{c_1}$ and let $p= v_0 v_1 \dots v_t$ be a path in $\Delta$. Abusing the notation and identifying the vertices with their images in $\Gamma^e$ we have that $v_i=a_{j_i}^{w_{j_i}}$, $i=0, \dots, t-1$, $v_t=a_2^{w_2}$ and $v_i \ne v_t$ for $i\ne t$. Since $p$ is a path in the commutation graph $\Delta$, it follows that $[v_i,v_{i+1}]=1$, $i=0, \dots, t-1$ and so $[a_{j_i}^{u_{j_i}}, a_{j_{i+1}}^{u_{j_{i+1}}}]=1$. Since $v_0 \in S_{c_1}$, it follows from the definition of $S_{c_1}$ that $a_{j_0}^{u_{j_0}}=a_{j_0}^{u_{j_0}'c_1}$. Since $[a_{j_0}^{u_{j_0}'c_1}, a_{j_1}^{u_{j_1}}]=1$ it follows from Lemma \ref{claim1} that either $c_1$ right-divides $u_{j_1}$ (and so $v_1\in S_{c_1}$) or $c_1$ disjointly commutes with $a_{j_1}^{u_{i_1}}$ (and so $v_{1} \notin S_{c_1}$). Let $r$ be so that for all $0\le r'<r$, $v_{r'}\in S_{c_1}$ and $v_r \notin S_{c_1}$. Note that such $r>0$ exists because $v_t=a_2^{w_2} \notin S_{c_1}$. We show that $0<r<t$ and so the statement follows. Indeed assume towards a contradiction that $v_{t-1} \in S_{c_1}$, i.e. $r=t$. Since $[a_{j_{t-1}}^{u_{j_{t-1}}'c_1}, a_2]=1$ and $c_1 $ does not right-divide $u_2$ ($u_2=1$), it follows from Lemma \ref{claim1} that $c_1$ disjointly commutes with $a_2$ - contradiction with the choice of $c_1$.

\medskip

For each vertex $a_i^{w_i}$ in $S_{c_1}$, we define $q(i,1)$ to be 1 and for each vertex $a_i^{w_i} \notin S_{c_1}$, we define $q(i,1)$ to be $0$. Define the graph $\Delta^{1}$ to be the extended commutation graph of 

$$
\{ a_i^{v_i'} \mid v_i'=c_1 \hbox{ if } v_i=a_i^{w_i} \in S_{c_1} \hbox{ and } v_i'=1 \hbox{ otherwise} \}.
$$

\medskip

\begin{claim}\label{claim3} 
The induction hypothesis {\rm(IH1)} holds: $E(\Delta) \subset E(\Delta^{1})$, that is, if $[a_i^{u_i}, a_j^{u_j}]=1$ then $[a_i^{v_i'}, a_j^{v_j'}]=1$.
\end{claim} 

If $[a_i^{u_i}, a_j^{u_j}]=1$, then, in particular, we have that $[a_i, a_j]=1$ and so the statement is clear if both $a_i^{u_i}$ and $a_j^{u_j}$ simultaneously belong or do not belong to $S_{c_1}$. The case $a_i^{u_i} \in S_{c_1}$ and $a_j^{u_j} \notin S_{c_1}$ follows from Lemma \ref{claim1}.

\medskip

Furthermore, it follows from the construction that the occurrence $c_1$ cancels in $v_i' {v_j'}^{-1}$ if and only if $q(i,1)=q(j,1)=1$ if and only if the marked occurrences of $c_1$ in $u_i$ and $u_j$ cancel in $u_iu_j^{-1}$ (and the former occurs if and only if $a_i^{w_i}, a_j^{w_j} \in S_{c_1}$). Therefore the induction hypothesis (IH2) holds.

Since by definition $a_1 c_1 a_2$ defines a path in the non-commutation graph $\overline \Delta$, we have that $[a_1^{v_1'},a_2^{v_2'}]=[a_1^{c_1}, a_2] \ne 1$. Therefore, $E(\Delta^{1}) \subsetneq E(\Delta')$ and so the induction hypothesis (IH3) is satisfied.

\medskip

Base of induction for $s$ at step $m$, that is $u_1= r = c_m \cdots g_1 c_1$ and $s=1$. In this case, the proofs of the induction hypothesis (IH1) and (IH2) are analogous to the corresponding proofs in the base of induction $m=s=1$.

\medskip

Define $S_{c_s}$ to be the set of vertices of $\Delta$ such that $c_s$ right-divides $u_i c_1^{-1} g_1^{-1} \cdots c_{s-1}^{-1} g_{s-1}^{-1}$. In particular, $v_1 \in S_{c_s}$. For each vertex $a_i^{u_i}$ in $S_{c_s}$, we mark the occurrence of $c_s$ in $u_i$ that right-divides $u_i c_1^{-1} g_1^{-1} \cdots c_{s-1}^{-1} g_{s-1}^{-1}$.

Let $\Delta^{s}$ be the extended commutation graph of the set $\{ a_i^{v_{i,s}'}\}$, where 
$$
v_i' = v_{i,s}'= \left\{
\begin{array}{ll}  
c_j \dots c_1,& \hbox{if both } a_i^{w_i} \in S_{c_j} \hbox{ and } a_i^{w_i} \notin S_{c_{j+1}}, j=1, \dots, s;\\
1, & \hbox{if } a_i^{w_i} \notin S_{c_1}.
\end{array}\right.
$$
In the first case we have $q(i,s)=j$ and in the second one, $q(i,s)=0$. By convention, $S_{c_{s+1}}=\emptyset$.

Assume by induction that for all $s<m$ the following condition holds:
$S_{c_s} \subset S_{c_{s-1}} \subset \dots \subset S_{c_1}$.

Let us prove the induction hypothesis for $s\le m$ at step $m$. The proofs for $s < m$ and $s=m$ are slightly different, the latter being more general as it requires proof of (IH3). Hence, we further assume that $s=m$.

Recall that $S_{c_m}$ is the set of vertices of $\Delta$ such that $c_m$ right-divides $u_i c_1^{-1} g_1^{-1} \cdots c_{m-1}^{-1} g_{m-1}^{-1}$. In particular, $v_1 \in S_{c_m}$. 

\medskip

\begin{claim} \label{claim4}
$S_{c_m} \subset S_{c_{m-1}} \subset \ldots \subset S_{c_1}$.
\end{claim}

Indeed, if $a_i^{w_i} \notin S_{c_{m-1}}$, then from the definition we have that $c_{m-1}$ does not cancel in $u_i c_1^{-1} g_1^{-1} \cdots g_{m-1}^{-1}c_{m-1}^{-1}$. Then, since $c_{m-1}$ and $c_m$ do not commute, it follows that $c_m$ does not cancel in $u_i c_1^{-1} g_1^{-1} \cdots c_{m-1}^{-1} g_{m-1}^{-1}c_m^{-1}$; therefore, $c_m$ does not right-divide $u_i c_1^{-1} g_1^{-1} \cdots c_{m-1}^{-1} g_{m-1}^{-1}$ and so $a_i^{w_i} \notin S_{c_m}$.

\medskip

Recall that $\Delta^{m}$ is the extended commutation graph with the vertex set $\{ a_i^{v_i'}\}$, where $v_i'=v_{i,m}'= c_j \dots c_1$ if both $a_i^{w_i} \in S_{c_j}$ and $a_i^{w_i} \notin S_{c_{j+1}}$, $j=1, \dots, m$ ($S_{c_{m+1}}=\emptyset$) and $v_i'=1$ if $a_i^{w_i} \notin S_{c_1}$.

\begin{claim} \label{claim5}
The induction hypothesis (IH1) holds: $E(\Delta) \subset E(\Delta^{m})$. 
\end{claim}

The statement follows immediately by induction if both $a_i^{w_i}$ and $a_j^{w_j}$ simultaneously belong or do not belong to $S_{c_m}$. Assume without loss of generality that $a_i^{w_i} \in S_{c_m}$, $a_j^{w_j} \in S_{c_s} \setminus S_{c_m}$ and $(a_i^{w_i}, a_j^{w_j}) \in E(\Delta)$. Since by induction $E(\Delta) \subset E(\Delta^{m-1})$, it follows that $(a_i^{v_i'}, a_j^{v_j'}) \in E(\Delta^{m-1})$, i.e. $[a_i^{c_{m-1} \cdots c_1}, a_j^{c_s \cdots c_1}]=1$. Again by induction, the marked occurrences cancel in $u_iu_j^{-1}$ if and only if they do in $v_i' {v_j'}^{-1}$. Hence, it follows that the marked occurrence $c_{s+1}$ does not cancel in $u_iu_j^{-1}$.

Since $[a_i^{w_i}, a_j^{w_j}]=1$, it follows that $[a_i^{u_i}, a_j^{u_j}]=1$ and so $[a_i^{u_iu_j^{-1}}, a_j]=1$. As we argued above, the marked occurrence $c_{s+1}$ does not cancel in $u_iu_j^{-1}$. Since $c_{s+1}, \dots, c_m$ is a path in the non-commutation graph $\overline \Delta$, we conclude that the marked occurrences $c_{s+1}, \dots, c_m$ do not cancel in $u_iu_j^{-1}$ and so they appear in the geodesic of $u_iu_j^{-1}$. It now follows from Lemma \ref{claim1}, that $c_m$ disjointly commutes with $a_j$. Since by induction hypothesis, we have that $E(\Delta) \subset E(\Delta^{m-1})$ and $c_m$ disjointly commutes with $a_j$, we conclude that $[a_i^{v_i'}, a_j^{v_j'}]=[a_i^{c_m \cdots c_1}, a_j^{c_s \cdots c_1}]=1$.

\medskip

Note that by construction $v_1'=c_m \cdots c_1$ and $v_2'=1$. We conclude that $[a_1^{v_1'}, a_2^{v_2'}] \ne 1$ and so $(a_1^{v_1'}, a_2^{v_2'}) \in E(\Delta') \setminus E(\Delta^{m})$. This proves that the induction hypothesis (IH3) holds for $s=m$.

We are left to show that the induction hypothesis (IH2) holds:

\begin{claim}\label{claim6}
If $a_i^{w_i} \in S_{c_s}$, then $u_i=u_i' c_s h_{s-1}^{(i)} c_{s-1} \cdots c_2 h_1^{(i)} c_1$ and the marked occurrences of $c_p$ cancel in $u_iu_j^{-1}$ if and only if they cancel in $v_i' {v_j'}^{-1}$ if and only if $p\le\min\{q(i,s),q(j,s)\}$.
\end{claim}

Let us first show that if $a_i^{w_i} \in S_{c_m}$, then $u_i=u_i'' c_m h_{m-1}^{(i)} c_{m-1} \cdots c_2 h_1^{(i)} c_1$.

Suppose that $a_i^{w_i} \in S_{c_m} \subset S_{c_{m-1}}$. It follows by induction that 
$$
u_i=u_i' c_{m-1} h_{m-2}^{(i)}\cdots c_2 h_1^{(i)} c_1.
$$ 

In this case, by construction we have that in $\Delta^{m-1}$, $v_i'=c_{m-1} \dots c_1$ and $v_1'=c_{m-1}\cdots c_1$. Since by induction the occurrences of $c_1, \dots, c_{m-1}$ cancel in $v_i'{v_1'}^{-1}$, then they also cancel in $u_i u_1^{-1}$, i.e. the occurrences $c_i$, $i=1, \dots, m-1$, cancel in 
$$u_i'c_{m-1} \cdots c_2 h_1^{(i)} c_1 c_1^{-1} g_1^{-1} c_2 \cdots c_{m-1}^{-1} g_m^{-1} c_m^{-1}.
$$
This implies that 
$$[c_k, h_{k-1}^{(i)} \cdots h_1^{(i)} g_1^{-1} \cdots g_{k-1}^{-1}]=1, \ k=2, \dots, m-1.
$$
In particular, since $[c_{m-1}, h_{m-1}^{(i)} \cdots h_1^{(i)} g_1^{-1} \cdots g_{m-1}^{-1}]=1$ and $[c_{m-1}, c_m] \ne 1$, it follows that $c_m$ does not occur in the geodesic $h_{m-1}^{(i)} \cdots h_1^{(i)} g_1^{-1} \cdots g_{m-1}^{-1}$.

Since $a_i^{w_i} \in S_{c_m}$, it follows from the definition that $c_m$ right-divides $u_i c_1^{-1} \cdots g_{m-1}^{-1}$; from the above discussion, it follows that $c_m$ right-divides $u_i' h_{m-1}^{(i)} \cdots h_1^{(i)} g_1^{-1} \cdots g_{m-1}^{-1}$ but since $c_m$ does not occur in the geodesic 
$h_{m-1}^{(i)} \cdots h_1^{(i)} g_1^{-1} \cdots g_{m-1}^{-1}$, it follows that $c_m$ right-divides $u_i'$, i.e. $u_i = u_i'' c_m h_{m-1}^{(i)} \cdots h_1^{(i)}$. 

\smallskip

Finally, let us  prove that the marked occurrences $c_p$ cancel in the product $u_i u_j^{-1}$ if and only if they cancel in the product $v_i' {v_j'}^{-1}$. Using the induction hypothesis, it suffices to prove that the marked occurrence of $c_m$ cancels in $u_iu_j^{-1}$ if and only if it does in $v_i' {v_j'}^{-1}$. By definition, the occurrence $c_m$ cancels in $v_i {v_j'}^{-1}$ if and only if $a_i^{w_i}, a_j^{w_j} \in S_{c_m}$. Hence, we have to show that $a_i^{w_i}, a_j^{w_j} \in S_{c_m}$ if and only if $c_m$ cancels in $u_iu_j^{-1}$. 

As argued before, if  $a_i^{w_i} \in S_{c_m}$, then $c_m$ right-divides 
$$
u_i'' c_m h_{m-1}^{(i)} \cdots h_1^{(i)} g_1^{-1} \cdots g_{m-1}^{-1}
$$ 
and so 
$$
[c_m, h_{m-1}^{(i)} \cdots h_1^{(i)} g_1^{-1} \cdots g_{m-1}^{-1}]=1.
$$

We have 
\begin{equation}
\begin{split}
& u_iu_j^{-1} = (u_i c_1^{-1} g_2^{-1} \cdots g_{m-1}^{-1}) {(u_j c_1^{-1} g_2^{-1} \cdots g_{m-1}^{-1})}^{-1}=\\
& u_i'' c_m h_{m-1}^{(i)} \cdots h_1^{(i)} g_1^{-1} \cdots g_{m-1}^{-1} g_{m-1} \cdots g_1 {h_1^{(j)}}^{-1} \cdots {h_{m-1}^{(j)}}^{-1} c_m^{-1} u_j''
\end{split}
\end{equation}
and since 
$$
[c_m, h_{m-1}^{(k)} \cdots h_1^{(k)} g_1^{-1} \cdots g_{m-1}^{-1}]=1,\ k=i,j,
$$
it follows that the marked occurrence of $c_m$ cancels. Conversely, if $c_m$ cancels, we have that 
$$
[c_m, h_{m-1}^{(i)} \cdots h_1^{(i)} g_1^{-1} \cdots g_{m-1}^{-1}]=1
$$ 
hence it right-divides $u_i c_1^{-1} \cdots g_{m-1}^{-1}$ and so $a_i^{w_i}$ belongs to $S_{c_m}$.

This proves the statement $\Phi(k,1)$.

\medskip

Let us prove the statement $\Phi(k,k')$. Recall that $E(\Delta') \setminus E(\Delta) =\{e_1, \dots, e_k\}$ and that $\Delta'$ is the extended commutation graph of $\{ a_i^{v_i'}\}$, where $v_i'$ are words composed from the (ordered) set of marked letters in $w_i$, $i=1, \dots, n$. By induction, assume that (IH1), (IH2) and (IH3) are satisfied for $l <k'$. Without loss of generality, assume that $e_{k'}=(a_1,a_2)$.

\smallskip

Since $(a_1^{w_1}, a_2^{w_2}) \notin E(\Delta)$, it follows that $[a_1^{w_1}, a_2^{w_2}]\ne 1$ and so $[a_1^{w_1w_2^{-1}}, a_2] \ne 1$. On the other hand, since $(a_1^{v_1'}, a_2^{v_2'}) \in E(\Delta')$, it follows that $[a_1^{v_1'}, a_2^{v_2'}]=1$ and so $[a_1^{v_1'{v_2'}^{-1}}, a_2] =1$. Since marked letters that do not cancel in $v_1'{v_2'}^{-1}$ (and so by induction, they do not cancel in $w_1w_2^{-1}$) commute with $a_2$, it follows that the double-coset representative $C(a_1) w_1w_2^{-1} C(a_2)$ of $w_1w_2^{-1}$ must contain a non-marked letter. This means that if  $w_1w_2^{-1}=D_1 r D_2$ where $D_i$ is the maximal divisor of $w_1w_2^{-1}$ from $C(a_i)$, then $r$ contains a non-marked letter. As in the base case, there exist $c_m, \dots, c_1$ such that $r= g_m c_m \cdots g_1 c_1 g_0$, $c_1$ right-divides $r$, $c_m$ left-divides $r$, the tuple $(a_1, c_m, \dots, c_1, a_2)$ defines a path in the non-commutation graph $\overline \Gamma$, $c_i \ne c_j$ for all $i\ne j$, and at least one $c_i$ is not marked. Among such paths, we choose a shortest one and mark the right most occurrences of the letters $c_i$ in $r$, $i=1, \dots, m-1$. Since there are no repetitions of letters in the path, it follows that we have marked at most $m$ new letters and so $m \le |V(\Gamma)|$. Since by induction, the number of marked letters up to step $l=k'-1$ is less than $(k'-1)|V(\Gamma)|$, it follows that at step $k'$ there are at most $k'|V(\Gamma)|$ marked letters.

Let $u_i a_i u_i^{-1}$ be a geodesic word corresponding to the element $a_i^{w_i w_2^{-1} D_2^{-1}}$. As in the base case, we have that $u_2=1$ and $u_1=r$, $c_1$ right-divides $u_1$ and so $u_1=u_1' c_1$. Similarly, if $v_1'{v_2'}^{-1}= d_1 t d_2$, where $d_i$ is the maximal divisor from $C(a_i)$, we define $t_i$ so that the word $a_i^{v_i' {v_2'}^{-1}d_2^{-1}}$ is geodesic. By induction, it follows that $t_i$ are words defined by the marked letters in $u_i$ and so in particular $t_2=1$.

As in the proof of statement $\Phi(k,1)$, we proceed by induction on $m$ and $s$.

We establish the base of induction: we consider simultaneously the cases $m=s=1$ and $s=1$ at step $m$.

Let $S_{c_1}$ be the set of vertices of $\Delta$ such that $c_1$ right-divides $u_i$. If $c_1$ is a marked occurrence, we proceed to $c_2$. If $c_1$ is not a marked occurrence, we mark it at this step.

As in the case $\Phi(k,1)$, see Claim \ref{claim2}, any path in $\Delta$ connecting a vertex in $S_{c_1}$ to the vertex $a_2^{w_2}$ contains a vertex $v$ such that $v\ne a_2^{w_2}$ and $v \notin S_{c_1}$.

Define the graph $\Delta^{1}$ to be the extended commutation graph of the set $\{ a_i^{w_i'}\}$, where  $w_i'$ is the word defined by the marked occurrences, that is $w_i'=v_i'$ if  $a_i^{w_i} \notin S_{c_1}$ and $w_i'=v_{i1}' c_1 v_{i_2}'$ otherwise, here $v_i'=v_{i1}'v_{i2}'$ and $c_1$ right-divides $w_i'$.

As in the case $\Phi(k,1)$, see Claim 2, we have that $E(\Delta) \subset E(\Delta^{1})$, i.e. if $[a_i^{u_i}, a_j^{u_j}]=1$, then $[a_i^{w_i'}, a_j^{w_j'}]=1$. Furthermore, it follows from the construction that the occurrence $c_1$ cancels in $w_i' {w_j'}^{-1}$ if and only if the corresponding occurrence cancels in $u_iu_j^{-1}$ (and the former occurs if and only if $a_i^{w_i}, a_j^{w_j} \in S_{c_1}$). Therefore the induction hypothesis (IH1) and (IH2) hold.

Recall that $u_1=r= c_m \cdots g_2 c_2 g_1 c_1$. If $m=1$, then by definition $a_1 c_1 a_2$ defines a path in the non-commutation graph $\overline \Delta$. In this case, we have that $[a_1^{w_1'},a_2^{w_2'}]=[a_1^{v_{11}'c_1v_{12}'}, a_2] \ne 1$, or equivalently, $E(\Delta^{1}) \subsetneq E(\Delta')$ and so the induction hypothesis (IH3) is also satisfied.

Let us prove the induction hypothesis for $s\le m$ at step $m$. As in the base case, we assume that $s=m$.

Let $S_{c_m}$ be the set of vertices of $\Delta$ such that $c_m$ right-divides $u_i c_1^{-1} g_1^{-1} \cdots c_{m-1}^{-1} g_{m-1}^{-1}$. In particular, $v_1 \in S_{c_m}$. As in the case $\Phi(k,1)$, see Claim 3, $S_{c_m} \subset S_{c_{m-1}} \subset \dots \subset S_{c_1}$.

Let $\Delta^{m}$ be the extended commutation graph induced of $\{ a_i^{w_i'}\}$, where $w_i'$ are defined as follows. If $c_m$ is an occurrence marked in the previous steps, then we define $w_i'= v_i'$ where $v_i'$ is the conjugator in $\Delta^{m-1}$; if $c_m$ is not a marked occurrence, we mark it and define $w_i'=v_{i1}'c_m v_{i2}'$ if $a_i^{w_i} \in S_{c_m}$; and otherwise, $w_i'=v_i'$ where $v_i'=v_{i1}'v_{i2}'$.

We now show that $E(\Delta) \subset E(\Delta^{m})$. The statement follows immediately by induction if both $a_i^{w_i}$ and $a_j^{w_j}$ simultaneously belong or do not belong to $S_{c_m}$ or if $c_m$ was a marked occurrence. Assume without loss of generality that $c_m$ was not marked, $a_i^{w_i} \in S_{c_m}$, $a_j^{w_j} \in S_{c_s} \setminus S_{c_m}$ and $(a_i^{w_i}, a_j^{w_j}) \in E(\Delta)$. Since by induction $E(\Delta) \subset E(\Delta^{m-1})$, it follows that $(a_i^{v_i'}, a_j^{v_j'}) \in E(\Delta^{m-1})$. Again by induction, we have that marked occurrences cancel in $u_iu_j^{-1}$ if and only if they do in $v_i' {v_j'}^{-1}$. Hence, it follows that the marked occurrence $c_{s+1}$ does not cancel in $u_iu_j^{-1}$.

Since $[a_i^{w_i}, a_j^{w_j}]=1$ it follows that $[a_i^{u_i}, a_j^{u_j}]=1$ and so $[a_i^{u_iu_j^{-1}}, a_j]=1$. As we argued above, the marked occurrence $c_{s+1}$ does not cancel in $u_iu_j^{-1}$. Since $c_{s+1}, \dots, c_m$ is a path in the non-commutation graph $\overline \Delta$, we conclude that the marked occurrences $c_{s+1}, \dots, c_m$ do not cancel in $u_iu_j^{-1}$ and so, in particular, they appear in the geodesic of $u_iu_j^{-1}$. It follows from Lemma \ref{claim1}, that $c_m$ disjointly commutes with $a_j$. Since by induction hypothesis, we have that $E(\Delta) \subset E(\Delta^{m-1})$ and $c_m$ disjointly commutes with $a_j$, we conclude that $[a_i^{v_i'}, a_j^{v_j'}]=1$. Hence, the induction hypothesis (IH1) holds.

Furthermore, since by construction $w_1'=v_{im}'d_m \cdots d_1 v_{i0}'$, where $v_i'=v_{im}' \cdots v_{i0}'$, $d_i$ are the new marked letters and $v_2'=1$, we have that $[a_1^{w_1'}, a_2^{w_2'}] \ne 1$ and so $(a_1^{w_1'}, a_2^{w_2'}) \in E(\Delta') \setminus E(\Delta^{m})$. The induction hypothesis (IH3) is satisfied.

As in the case $\Phi(k,1)$, we have that if $a_i^{w_i} \in S_{c_m}$, then $u_i=u_i'' c_m h_{m-1}^{(i)} c_{m-1} \cdots c_2 h_1^{(i)} c_1$.

Let us finally prove that the marked occurrences of $c_i$ cancel in the product $u_i u_j^{-1}$ if and only if they cancel in the product $v_i' {v_j'}^{-1}$. Using the induction hypothesis, it suffices to prove that the new marked occurrence of $c_m$ cancels in $u_iu_j^{-1}$ if and only if it does in $v_i' {v_j'}^{-1}$. By construction, the occurrence $c_m$ cancels in $v_i {v_j'}^{-1}$ if and only if $a_i^{w_i}, a_j^{w_j} \in S_{c_m}$. Hence, we have to show that $a_i^{w_i}, a_j^{w_j} \in S_{c_m}$ if and only if $c_m$ cancels in $u_iu_j^{-1}$. The proof is analogous to that of Claim \ref{claim6}. This shows that the induction hypothesis (IH2) holds.

\medskip

Since $|E(\Delta') \setminus E(\Delta)| \le \frac{n (n-1)}{2}$, it follows that there are most $k=\frac{n (n-1)}{2}$ steps; since at each step we mark at most $M=|V(\Gamma)|$ letters, we conclude that $|w_i'| \le \frac{n (n-1) M }{2} < n^2 M$. Therefore, if the graph $\Delta$ is connected (the number of connected components $K$ is 1), we have that $\Delta < B(\Gamma^e, R')$, where $R'= n^2 M$. 

\medskip

Let $\Delta$ be the disjoint union of $K$ connected graphs $\Lambda_1, \dots, \Lambda_K$. We prove the statement by induction on $K$. If $K=1$, then the graph $\Delta$ is connected and the statement follows from the above discussion. Let $\Delta_1$ be the disjoint union of $K-1$ connected components $\Lambda_i$, $i=1, \dots, K-1$ and $\Delta_2 = \Lambda_K$. 

If $\Gamma$ is a join, then $\GG(\Gamma)$ is the direct product, namely $\GG(\Gamma) = \GG(\Gamma_1) \times \GG(\Gamma_2)$. In this case $\Gamma^e$ is also a join of $\Gamma_1^e$ and $\Gamma_2^e$. Consider the decomposition of the set of vertices $V(\Delta)$ as the union of $S_1$ and $S_2$ such that the graph embedding $f: \Delta \to \Gamma^e$ maps $f(S_1) \subset V(\Gamma_1^e)$ and $f(S_2) \subset V(\Gamma_2^e)$. Since $\Gamma^e$ is a join, so is the graph induced by $f(S_1) \cup f(S_2)$ in $\Gamma^e$ (provided that  $S_1$ and $S_2$ are non-empty). Since by assumption $\Delta=S_1 \cup S_2$ is disconnected and so, in particular, it is not a join, it follows that either $S_1$ or $S_2$ is empty. Therefore $\Delta$ embeds in either $\Gamma_1$ or $\Gamma_2$. Therefore, without loss of generality we can assume that $\Gamma$ is not a join. 

Assume by induction that the lemma holds for graphs with less than $K$ connected components, i.e. $\Delta_i$ embeds into the ball of radius $R_i$ in $\Gamma^e$, $i=1,2$, correspondingly. Let $f_1$ and $f_2$ be the corresponding embeddings. Define $f: \Delta \to \Gamma^e$ as follows:  
$$
f(v)= \left\{ 
\begin{array}{ll}
f_1(v),& \hbox{if } v\in \Delta_1 \\
f_2(v)^w, & \hbox{if } v \in \Delta_2,
\end{array} \right.
$$
where $w= (y_1\cdots y_M)^{(M+R_1+R_2)}$. Note that $f$ is well-defined since the conjugate $f_2(v)^w$ is again a conjugate of a generator and so it is a vertex in $\Gamma^e$. We claim that $f$ is an embedding. Indeed, since by assumption $f_1$ and $f_2$ are graph embeddings, it suffices to show that the set of vertices $V(f(\Delta_1)$ and $V(f(\Delta_2))$ are disjoint, that is $V(f(\Delta_1)) \cap V(f(\Delta_2))=\emptyset$, and that there are no edges between vertices in $f(\Delta_1)$ and vertices in $f(\Delta_2)$. 

Let $y_i^{v_i}$ be the image of a vertex in $\Delta_1$ and $y_j^{v_jw}$ be the image of a vertex in $\Delta_2$. Since by assumption $y_i^{v_i}$ and $y_j^{v_j}$ belong to the balls of radius $R_1$ and $R_2$ respectively and the length of $w$ is $M(M +R_1+R_2)$, it follows that $(y_1\cdots y_M)^M$ is a subword of the cyclically reduced part of $v_jwv_i^{-1}$, i.e. $v_jwv_i^{-1} = u_1 (y_1\cdots y_M)^M u_2$. It now follows from \cite{CK1} that $(y_1 \cdots y_M)$ is a subword of the geodesic of $(y_j^{u_1})^{(y_1\cdots y_M)^M}$ and so it is a subword of the geodesic of $y_j^{v_jwv_i^{-1}}$. Hence, $y_i \ne y_j^{v_jwv_i^{-1}}$ and so $V(f(\Delta_1)) \cap V(f(\Delta_2))=\emptyset$. 

By definition, a vertex $y_i^{v_i}$ in the image of $\Delta_1$ is connected by an edge to a vertex $y_j^{v_jw}$ in the image of $\Delta_2$ if and only if they commute, that is $[y_i^{v_i}, y_j^{v_jw}]=1$ in $\GG(\Gamma)$ or, equivalently, $[y_i, y_j^{v_jwv_i^{-1}}]=1$. However, it follows from the description of centralisers of elements in a pc group, see Theorem \ref{thm:centr}, that if $[y_i, y_j^{v_jwv_i^{-1}}]=1$ and $(y_1 \cdots y_M)$ is a subword of the geodesic of $y_j^{v_jwv_i^{-1}}$, then $y_1, \dots, y_M \in C(y_i)$. Since $y_1, \dots, y_M$ generate $\GG(\Gamma)$, if $y_1, \dots, y_M \in C(y_i)$, then $y_i$ is central - this contradicts the fact that $\Gamma$ is not a join. We deduce that $[v_jwv_i^{-1}, y_i]\ne 1$ and so $\Delta$ embeds into the ball of radius $M(M+R_1+R_2)+R_2$. By induction, we have that $R_1 < 2(K-1)n^2M^K$; since $\Delta_2$ is connected, we have shown that $R_2 < n^2M$. In particular, we have that $R_1 + R_2 < (2K-1) n^2 M^K$. We conclude that $R< M(M+R_1 +R_2)+R_2 < 4Kn^2 M^{K+1}$ and $\Delta$ embeds in $B(\Gamma^e, R)$. This finishes the proof of the theorem.

\end{proof}

\begin{cor}\label{cor:decEGE}
There is an algorithm that given two simplicial graphs $\Delta$ and $\Gamma$ decides whether or not $\Delta$ is an induced subgraph of $\Gamma^e$.
\end{cor}

\bigskip

The Embedding Problem (coding) between partially commutative monoids (trace monoids) is particularly interesting in the context of concurrent systems and the existence of uniquely decipherable morphisms and it was intensively studied in the 80's. In \cite{Kunc}, Kunc proves that the Coding Problem for pc monoids is undecidable, or in other words, there is no algorithm to decide whether or not there exists an embedding between two given trace monoids. However, if one restricts to the case when the defining graph of the source pc monoid is the complement of a forest, then the Embedding Problem is decidable.

The decidability of the (Tame) Embedding Problem for partially commutative groups is still an open question, see Question \ref{q4}. We will return to this question in Section \ref{sec:universal}, where we present some facts that bring us to conjecture, see Conjecture \ref{conj:dectameemb}, that the Tame Embedding Problem for pc groups is decidable. 

\medskip

A closely related problem on embeddability asks if a given map between partially commutative groups is an embedding. In the case of monoids this is a well-known problem called the \emph{trace code problem}. The trace code problem for pc monoids is known to be undecidable even for tame morphisms, see \cite{Dikert-Metivier}. In the context of partially commutative groups an analogous question is also open.
\begin{question}
Does there exist an algorithm to decide whether or not a given map between partially commutative groups is an embedding?
\end{question}

The above question in the context of groups is more subtle than the trace code problem for monoids. It is well-known, see \cite{Dikert-Metivier}, that there is no algorithm to decide whether or not a morphism from a free monoid to the direct product of free monoids is an embedding. However, in the case of groups such an algorithm exists. Notice that a homomorphism $\varphi$ from the free group $F_n$ to the direct product of free groups $F_n^{(1)}\times F_n^{(2)}$ is an embedding if and only if $\varphi \pi_i: F_n \to F_n^{(i)}$ is an embedding for (at least one) $i \in \{1,2\}$, where $\pi_i$ is the canonical projection from $F_n^{(1)} \times F_n^{(2)}$ to its factor $F_n^{(i)}$. Indeed, assume that both homomorphisms $\varphi \pi_i$ are not injective, $i=1,2$. Let $w_i \in F_n$ be such that $w_i \ne 1$ and $\varphi\pi_i(w_i)=1$, $i=1,2$. If $w_1 = w_2$, then $\varphi(w_1)=(1,1)$ and so $\varphi$ is not injective. If $w_1 \ne w_2$, we have that $[w_1, w_2] \ne 1$, but 
$$
\varphi([w_1,w_2])=[(1,\varphi\pi_2(w_1)), (\varphi\pi_1(w_2),1)]=1
$$ and so $\varphi$ is not injective. In turn, $\varphi\pi_i$ is an embedding if and only if the image is a free group of rank $n$ and this can be decided algorithmically.

\subsection{Conditions on the target pc group}

In the previous section, we proved that the Extension Graph Embedding Problem is decidable. This motivates one to look for conditions on the graphs for which the (Tame) Embedding and Graph Embedding Problems are equivalent, since one can then deduce the decidability of the (Tame) Embedding Problem. 

In this section we describe conditions on the target pc group that assure this equivalence. More concretely, we show that if the deflation graph of $\Gamma$ (see Definition \ref{def:defl}) is triangle-free then the Tame Embedding and the Extension Graph Embedding Problems  are equivalent. This result plays a crucial role in our proof that there is an algorithm to decide whether or not an arbitrary pc group is universally equivalent to a $2$-dimensional one (see Corollary \ref{cor:decunivth2dpc}).

\medskip

We begin by recalling an already known condition on the defining graph of the target pc group that assures the equivalence of the three Embedding Problems. The first result in this direction was proven by Kim and Koberda. In \cite{KK}, the authors show that the Embedding Problems are equivalent for $2$-dimensional pc groups, that is:

\begin{thm}[see \cite{KK}]\label{thm:KKtfree}
If $\Gamma$ is triangle-free, then the three Embedding Problems are equivalent, namely, there exists an embedding from $\GG(\Delta)$ to $\GG(\Gamma)$ if and only if there exists a tame embedding if and only if $\Delta$ is an induced subgraph of $\Gamma^e$.
\end{thm}

In \cite{CDK}, an analogous result was proven for triangle-built graphs, namely

\begin{thm}[see \cite{CDK}]\label{thm:CDKtbuilt}
If $\Gamma$ is triangle-built, then the three Embedding Problems are equivalent: there exists an embedding from $\GG(\Delta)$ to $\GG(\Gamma)$ if and only if there exists a tame embedding if and only if $\Delta$ is an induced subgraph of $\Gamma^e$.
\end{thm}

We now introduce two operations on a graph $\Gamma$: the $n$-inflation and the deflation of $\Gamma$. Informally, the $n$-inflation graph ``clones'' vertices of $\Gamma$ and so the associated pc group has the property that for every generator, there are $n-1$ different generators with the same centraliser. On the other hand, the deflation graph is a quotient graph that identifies vertices with the same star or, on the group level, it identifies generators with the same centraliser. More formally,
\begin{defn}[Inflation graph]\label{def:infl}
Let $\Gamma$ be a simplicial graph and $n \in \BN$. The $n$-\emph{inflation of} $\Gamma$, denoted by $\infl(\Gamma)$, is the following graph:
\begin{itemize}
\item for every vertex $v\in V(\Gamma)$ we introduce $n$-vertices $v_1, \dots, v_n \in V(\infl(\Gamma))$  and so $|V(\infl(\Gamma))|= n |V(\Gamma)|$; 
\item set $(v_i, v_j) \in E(\infl(\Gamma))$ for all $1\le i < j \le n$, i.e. the $n$-vertices associated to a vertex in $\Gamma$ define a clique in the inflation of $\Gamma$; we further set $(v_i, w_k) \in E(\infl(\Gamma))$ if and only if $(v,w) \in E(\Gamma)$. 
\end{itemize}
Obviously, the $1$-inflation graph coincides with the graph $\Gamma$. Furthermore, it is immediate that $\Gamma$ is an induced subgraph of its $n$-inflation.

The pc group $\GG(\infl(\Gamma))$ defined by the inflation graph is the graph product with the underlying graph $\Gamma$ and vertex groups isomorphic to the free abelian group of rank $n$.
\end{defn}

Recall that the graph product of groups is defined as follows. Let $G_i=\langle X_i\mid R_i\rangle$, $i=1,\dots, k$ be groups and let $\Gamma$ be a simplicial graph with vertices $v_1,\dots, v_k$. Then, the  graph product $mathcal G=\mathcal G(\Gamma;G_1,\dots,G_k)$ of groups $G_1,\dots, G_k$ with the underlying graph $\Gamma$, is a group with a presentation of the form
$$
\langle X_1,\dots, X_k\mid R_1,\dots, R_k, \mathcal{R}\rangle,
$$
where $\mathcal{R}=\{[x_i,x_j]\mid x_i \in X_i, x_j \in X_j \hbox{ and } (v_i,v_j) \in E(\Gamma)\}$.

\begin{defn}[Deflation graph]\label{def:defl}
Let $\Gamma$ be a simplicial graph. Given $v\in V(\Gamma)$, we define the \emph{star} $\St(v)$ of $v$ to be the set of vertices $\{v' \in V(\Gamma) \mid (v,v') \in E(\Gamma) \} \cup \{v\}$.

We say that two vertices $v,w\in V(\Gamma)$ are \emph{equivalent} if $\St(v)=\St(w)$ and write $v \sim w$. It is immediate to check that $\sim$ is an equivalence relation. On the level of groups, we have $v \sim w$ if and only if the centralisers $C(v)$ and $C(w)$ coincide.

The \emph{deflation of} $\Gamma$, denoted by $\defl(\Gamma)$ is the graph defined as follows:
\begin{itemize}
\item vertices in $\defl(\Gamma)$ are equivalence classes $[v]$ of vertices in $\Gamma$ defined by the equivalence relation $\sim$.
\item $([v],[w])$ is an edge in the deflation graph if and only if $(v,w)$ is an edge in $\Gamma$ (note that the definition does not depend on the choice of representative of the class).
\end{itemize}

It is straightforward to check that the map from $\Gamma$ to the deflation $\defl(\Gamma)$ is a retraction of graphs: there is an epimorphism from $\Gamma$ to $\defl(\Gamma)$ and $\defl(\Gamma)$ is an induced subgraph of $\Gamma$. Note that if a graph is connected so is its deflation and the deflation of a graph is never the graph with two vertices and one edge.
\end{defn}

It follows from the definitions that inflation and deflation are ``inverse'' operations, namely, $\defl (\infl( \Gamma)) = \defl (\Gamma)$.

In the next lemma we observe that the fact that the deflation graph of $\Gamma$ is triangle-free is a universal property.

\begin{lemma}\label{lem:universalprop}
Let $\Gamma$ be a simplicial graph. Then deflation graph $\defl(\Gamma)$ of $\Gamma$ is triangle-free if and only if $\GG(\Gamma)$ satisfies the following universal sentence:``for all non-trivial pair-wise commuting elements x,y,z, we have that $C(x)=C(y)$ or $C(x)=C(z)$'', that is
\begin{equation} \label{eq:property}
\begin{split}
&\forall x,y,z,t  \ [x,y]=1, [x,z]=1, [y,z]=1, x\ne 1, y \ne 1, z\ne 1 \\
& \rightarrow ( [t,x]= 1 \leftrightarrow [t,y]=1) \vee ([t,x]= 1 \leftrightarrow [t,z]=1)
\end{split}
\end{equation}
\end{lemma}

\begin{proof}
By definition the deflation graph of $\Gamma$ is triangle-free if and only if $\Gamma$ does not have three vertices with different centralisers that pair-wise commute. The latter is equivalent to saying that $\GG(\Gamma)$ satisfies the universal sentence (\ref{eq:property}).
\end{proof}

The main result of this section is the following
\begin{thm}\label{thm:defltfree}
If the deflation graph of $\Gamma$ is triangle-free, then the Tame Embedding Problem is equivalent to the Extension Graph Embedding Problem, i.e. there exists a tame embedding from $\GG(\Delta)$ to $\GG(\Gamma)$ if and only if $\Delta$ is an induced subgraph of the extension graph $\Gamma^e$.
\end{thm}

\begin{proof}

Assume that there exists a tame embedding $\varphi$ from $\GG(\Delta)$ to $\GG(\Gamma)$. We want to show that $\Delta < \Gamma^e$.

Since by assumption $\defl(\Gamma)$ is triangle-free, it follows from Lemma \ref{lem:universalprop} that $\GG(\Gamma)$ satisfies the universal sentence (\ref{eq:property}). Since $\varphi(\GG(\Delta)) < \GG(\Gamma)$, we have that $\Th_\forall( \GG(\Gamma)) \subset \Th_\forall(\Delta)$. Therefore $\GG(\Delta)$ satisfies property (\ref{eq:property}) and so by Lemma \ref{lem:universalprop} the deflation graph of $\Delta$ is also triangle-free.

\smallskip

If $\Delta$ is a clique, then $\GG(\Delta) < \GG(\Gamma)$ if and only if $\Gamma$ has a clique $\Gamma'$ as an induced subgraph and $|V(\Delta)|<|V(\Gamma')|$ and so $\Delta < \Gamma$. If $\Gamma$ is a clique, i.e. $\GG(\Gamma)$ is free abelian, then $\GG(\Delta)< \GG(\Gamma)$ if and only if $\Delta$ is a clique and $|V(\Delta)|< |V(\Gamma)|$, so $\Delta < \Gamma$.

Let us further assume that neither $\Delta$ nor $\Gamma$ is a clique. Since vertices in the same equivalence class in the deflation graph have the same centraliser and since the embedding is tame (there is no repetition of vertices of $\Gamma^e$ in the image), without loss of generality, one can assume that the image of a vertex is the product of vertices from different classes, that is $\varphi(v)=w_1 \cdots w_k$ and $[w_i] \ne [w_j]$ if $i \ne j$.

Assume towards contradiction, that there is a vertex, say $v_1$, such that its image is the product of vertices that belong to different classes. Since the deflation graph of $\Gamma$ is triangle-free, it follows that the image of $v_1$ is the product of at most two classes, that is $\varphi(v_1) = w_{11}  w_{12}$. 

Let $v_2$ be a vertex connected to $v_1$ so that $[v_1] \ne [v_2]$ (note that such $v_2$ exists since the deflation of $\Delta$ is connected and not a vertex). By property (\ref{eq:property}), since the deflation graph is triangle-free and since $v_1$ and $v_2$ commute, it follows that the image of $v_2$ is a product of vertices $w_{21}^\epsilon w_{22}^\delta$, $\epsilon, \delta \in \{0,1\}$ such that $[w_{11}]=[w_{21}]$ and $[w_{12}]=[w_{22}]$. Since $[v_1]\ne [v_2]$, it follows that either $\epsilon$ or $\delta$ is $0$. Without loss of generality, assume that $\delta=0$. Since $[v_1] \ne [v_2]$ and $C(\varphi(v_1)) < C(\varphi(v_2))$ and so $C(v_1) < C(v_2)$, it follows that the class of $[v_1]$ is a leaf in the deflation graph of $\Delta$ as $\defl(\Delta)$ is triangle-free. Furthermore, since $\defl(\Delta)$ is not the graph with two vertices and one edge, it follows that there is a vertex $v_3$ in $\Delta$ so that $v_3 \notin [v_1]$, $v_3 \notin [v_2]$ and $(v_3,v_2) \in E(\Delta)$.

Consider the subgraph of $\Delta^e$ induced by the set of vertices $W= (V(\Delta)\setminus \{v \in [v_1]\}) \cup \{v_3^{v} \mid v\in [v_1] \}$. Note that this graph is isomorphic to $\Delta$ since $[v_1]$ is a leaf in the deflation graph of $\Delta$ and is only connected to the vertex $[v_2]$ and since $[v_3]$ is also a vertex of $\defl(\Delta)$ connected to $[v_2]$. Hence, the subgroup $\GG(W)$ of $\GG(\Delta)$ defined by $W$ is isomorphic to $\GG(\Delta)$ and so the restriction of the embedding $\varphi$ to the subgroup defined by $W$ is again an embedding from a group isomorphic to $\GG(\Delta)$ to $\GG(\Gamma)$. Furthermore, since $[w_{11}]=[\varphi(v_2)]$ and $\varphi(v_2)$ and $\varphi(v_3)$ commute, we have that $\varphi(v_3^{v_1}) = \varphi(v_3) ^{w_{11} w_{12}} = \varphi(v_3)^{w_{12}}$. Therefore, the embedding $\varphi|_{\GG(W)}$ of $\GG(W)$ (a subgroup isomorphic to $\GG(\Delta)$) into $\GG(\Gamma)$ satisfies that $\varphi|_{\GG(W)}(v_1')$ is a vertex of $\Gamma^e$, $v_1'= v_3^{v_1}$ and so the number of vertices of $W \simeq \Delta$ whose image is a product of vertices from two different equivalence classes is reduced. Hence, proceeding by induction we conclude that there exists an embedding of $\Delta$ into $\Gamma^e$.
\end{proof}

\subsection{Conditions on the source group}

In this section we discuss conditions on the graph of the source pc group that assure that the three Embedding Problems are equivalent.

\medskip

In \cite{KK}, the authors show that a sufficient condition on the source group is that its graph be a forest.
\begin{thm}[see \cite{KK}]\label{thm:KKtree}
If $\Delta$ is a forest, then the three Embedding Problems are equivalent: there exists an embedding from $\GG(\Delta)$ to $\GG(\Gamma)$ if and only if there exists a tame embedding if and only if $\Delta < \Gamma^e$.
\end{thm}

We prove a ``dual'' result for the class $\mathcal{CT}$ of graphs whose non-commutation graph is a forest. More precisely, we show that if $\Delta$ is the complement of a forest, then there is an embedding from $\GG(\Delta)$ to $\GG(\Gamma)$ if and only if $\Delta$ is an induced subgraph of $\Gamma^e$. As a corollary, we deduce the decidability of the Embedding Problem when the defining graph of the source pc group is in $\mathcal{CT}$, or in other words, there is an algorithm that given a complement of a forest $\Delta$ and an arbitrary $\Gamma$ decides whether or not there exists an embedding from $\GG(\Delta)$ to $\GG(\Gamma)$. As we mentioned in the introduction, the Embedding Problem in pc monoids is undecidable. However, Kunc proves in \cite{Kunc} that if one requires the defining graph of the source pc monoid to be in $\mathcal{CT}$, then the problem is decidable. This reinforces the analogy between the results on decidability of the Embedding Problems for pc monoids and pc groups (although the techniques to prove them are rather different).

An important feature of the class $\mathcal{CT}$ is that it is ``generic'' in the sense that every pc group $\GG(\Delta)$ embeds into a pc group $\GG(\Gamma)$, for some $\Gamma \in \mathcal{CT}$, see \cite{KK2}. Roughly speaking, our result states that if a pc group defined by the complement of a tree embeds into a pc group, then it also embeds in it in a nice way, that is there exists an extension graph embedding. This seems to indicate that in order to understand whether or not the three Embedding Problems are equivalent for a given source pc group, it suffices to analyse how this pc group embeds into pc groups whose defining graph is in $\mathcal{CT}$.

\bigskip

We shall make use of the following
\begin{lemma}\label{claim1'}
Let $\GG(\Delta)$ be a pc group, $V(\Delta)=\{x_1,\dots, x_n\}$, and  let $\{x_{i_1}, \dots, x_{i_k} \}$ be a subset of pairwise different elements of $\{x_1, \dots, x_n\}$. Then, 
$$
[x_{i_1}, \dots, x_{i_k}] \ne 1
$$
if and only if each subset $\{x_{i_1}, x_{i_2} \}, \dots, \{x_{i_1}, \dots,x_{i_k} \}$ of vertices of the non-commutation graph $\overline \Delta$ spans a connected subgraph of $\overline \Delta$.
\end{lemma}
\begin{proof}
We first show that if a subset $\{x_{i_1}, \dots, x_{i_l} \}$ for some $l\in \{2, \dots, k\}$ spans a disconnected graph in $\overline \Delta$, then the commutator $[x_{i_1}, \dots, x_{i_k}]$ is trivial. Indeed, let $l$ be minimal so that the set $\{x_{i_1}, \dots, x_{i_l} \}$ spans a disconnected graph, i.e. $(x_{i_s}, x_{il})$ is not an edge in $\overline \Delta$, for all $1 \le s<l$ and so $[x_{i_s}, x_{i_l}]=1$ for all $1\le s <l$. It follows that $[x_{i_1}, \dots, x_{i_l}]=1$ and so $[x_{i_1}, \dots, x_{i_k}]=1$.

We use induction on $k$ to prove the converse. The induction hypothesis are as follows. If the above sets span connected subgraphs in $\overline \Delta$, then 
\begin{itemize}
\item $\alpha([x_{i_1}, \dots, x_{i_k}])=\{x_{i_1}, \dots, x_{i_k} \}$ and so in particular $[x_{i_1}, \dots, x_{i_k}]$ is non-trivial.
\item For all $a,b \in \alpha(c)$, where $c=[x_{i_1}, \dots, x_{i_k}]$ and any occurrence of $a$ in $c$, there is a decomposition of a cyclic permutation $cr(c)$ of $c$ of the type
$$
c_1 b^{\delta_1} c_2 a^{\epsilon_1} c_3 b^{\delta_2} c_4 a^{\epsilon_2},
$$
where $\delta_j, \epsilon_j\in \{\pm 1\}$, $j=1,2$. Furthermore, this decomposition satisfies the following properties:
\begin{itemize}
\item there is an ordered path from $a$ to $b$ in $\overline \Delta$ that can be read in $c_j$, $j=1,\dots, 4$; more precisely, 
$$
c_j=c_{j1}v_1c_{j1}v_2 \dots v_{s-1}c_{js_j}
$$ 
and we have the following paths in $\overline \Delta$: 
$$
p=a,v_1,\dots,v_{s_j-1}b \hbox{ if } j=1,3
$$ and 
$$
p'=bv_1,\dots, v_{s_j-1}a, \hbox{ if } j=2,4.
$$
\item $\alpha(c_j)$ or $\alpha(c_{j+1})$ contains $\alpha(c)\setminus \{a,b\}$.
\end{itemize}

\end{itemize}

\paragraph{Base of induction.} Assume that $k=2$. Since $\{x_{i_1}, x_{i_2}\}$ defines a connected subgraph, it follows that $[x_{i_1}, x_{i_2}]\ne 1$ and the induction hypothesis are satisfied.

\paragraph{Step of induction.} Assume that the sets $\{x_{i_1}, x_{i_2}\}, \dots, \{x_{i_1}, \dots, x_{i_k} \}$ span connected subgraphs. In particular $x_{i_k}$ is connected by an edge to at least one $x_{i_s}$, say $x_{i_{s_0}}$. Let us write the commutator $[x_{i_1}, \dots, x_{i_{k-1}}]$ as follows:
\begin{equation}\label{eq:commutator}
[x_{i_1}, \dots, x_{i_k}] =  [x_{i_1}, \dots, x_{i_{k-1}}]^{-1} x_{i_k}^{-1} [x_{i_1}, \dots, x_{i_{k-1}}] x_{i_k}.
\end{equation}

Since by assumption $x_{i_k} \notin \{x_{i_1}, \dots, x_{i_k-1}\}$, it follows that $x_{i_k}$ cancels in the commutator (\ref{eq:commutator}) if and only if it disjointly commutes with the geodesic of $[x_{i_1}, \dots, x_{i_{k-1}}]$. In turn, since by induction we have that $\alpha([x_{i_1}, \dots, x_{i_{k-1}}])=\{x_{i_1}, \dots, x_{i_{k-1}} \}$, we have that $x_{i_k}$ cancels if and only if $x_{i_k}$ disjointly commutes with $\{x_{i_1}, \dots, x_{i_{k-1}}\}$. Since $x_{i_k}$ is connected to $x_{i_{s_0}}$ in $\overline \Delta$, it follows that $x_{i_k}$ and $x_{i_{s_0}}$ do not disjointly commute. Therefore, $x_{i_k}$ does not cancel in $[x_{i_1}, \dots, x_{i_k}]$ and so $x_{i_k} \in \alpha([x_{i_1}, \dots, x_{i_k}])$.

Furthermore, by induction we have that for all $a,b \in \alpha([x_{i_1}, \dots, x_{i_{k-1}}])$ and any occurrence of $a$ there is a decomposition of (a cyclic permutation of $[x_{i_1}, \dots, x_{i_{k-1}}]$) of the type
\begin{equation}\label{eq:c4}
c_1 b^{\delta_1} c_2 a^{\epsilon_1} c_3 b^{\delta_2} c_4 a^{\epsilon_2},
\end{equation}
where $\delta_j, \epsilon_j\in \{\pm 1\}$, $j=1,2$, so that
\begin{itemize}
\item there is an ordered path from $a$ to $b$ in $\overline \Delta$ that can be read in $c_j$, $j=1,\dots, 4$;
\item $\alpha(c_j)$ or $\alpha(c_{j+1})$ contains $\alpha([x_{i_1}, \dots, x_{i_{k-1}}])\setminus \{a,b\}$. 
\end{itemize}

In particular, let $a$ be $x_{i_{s_0}}$ and $b \in \{x_{i_1}, \dots, x_{i_{k-1}}\} \setminus \{x_{i_{s_0}}\}$. Since $x_{i_{s_0}}$ does not commute with $x_{i_k}$, it cannot cancel in the commutator $c$, see Equation (\ref{eq:commutator}), and so $x_{i_{s_0}}$ occurs in the geodesic of $c$. Since by assumption there is an ordered path in $\overline \Delta$ from $a$ to $b$ that can be read in $c_4$ in the expression (\ref{eq:c4}) and since $a$ does not cancel, it follows that neither does any letter which occurs in the path. In particular, it follows that $b$ cannot cancel and so it appears in the geodesic of $c$. Therefore, we have that $\alpha([x_{i_1}, \dots, x_{i_k}])=\{x_{i_1}, \dots, x_{i_k} \}$. The decomposition of $c$ in the induction hypothesis is an immediate consequence of the decomposition of $[x_{i_1}, \dots, x_{i_{k-1}}]$. This finishes the proof.
\end{proof}

\begin{thm}\label{thm:noncommtree}
If $\overline \Delta$ is a forest, then the three Embedding Problems are equivalent: there exists an embedding from $\GG(\Delta)$ to $\GG(\Gamma)$ if and only if there exists a tame embedding if and only if $\Delta < \Gamma^e$.
\end{thm}
\begin{proof}
From Theorem 3 in \cite{KK}, if $\GG(\Delta)$ embeds into $\GG(\Gamma)$, then there exits an embedding $\varphi:\GG(\Delta) \to \GG(\Gamma)$ defined by the map $x_i \to y_{i1}\cdots y_{ir_i}$, where the set $\{ y_{i1}, \cdots, y_{ir_i}\}$ spans a clique with $r_i$ vertices in $\Gamma^e$. Let $\Lambda$ be the subgraph of $\Gamma^e$ induced by the vertices $\{y_{11}, \dots, y_{1r_1}, \dots, y_{n1}, \dots, y_{nr_n} \}$.

Our goal is to show that $\Delta$ is an induced subgraph of $\Lambda$ or, equivalently, that the complement graph $\overline{\Delta} $ is an induced graph of the complement graph $\overline \Lambda$.

\smallskip

We begin with an observation that we will use throughout the proof. 

\begin{remark}\label{rem:edges}
Firstly, for all $s,t \in \{1,\dots, r_i\}$, there is no edge $(y_{is}, y_{it}) \in \overline \Lambda$ since we have that $[y_{is}, y_{it}]=1$. Secondly, if $[x_i,x_j]=1$, then $[y_{i1} \cdots y_{ir_i}, y_{j1} \cdots y_{jr_j}]=1$ and in particular there are no edges $(y_{is}, y_{it})$ in $\overline{\Lambda}$, for all $1 \le s \le r_i$, $1\le t \le r_j$.
\end{remark}

Let $T$ be a connected component of the forest $\overline \Delta$. Since $T$ is a tree, we can order the vertices of $T$, $V=\{x_{i_1}, \dots, x_{i_k} \}$ so that $\{x_{i_1}, x_{i_2}\}, \dots, \{x_{i_1}, \dots, x_{i_k}\}$ span connected subgraphs of $\overline \Delta$. It follows from Lemma \ref{claim1'} that the commutator $c=[x_{i_1}, \dots, x_{i_k}]$ is non-trivial in $\GG(\Delta)$.

Recall that $\varphi(x_{i_l}) = y_{i_l,1}\cdots y_{i_l,r_{i_l}}$, $i_l=1, \dots,n$. Define $\mathcal C$ to be a maximal (by inclusion) connected subgraph of $\overline \Lambda$ obtained as union of subgraphs spanned by vertices $\{y_{i_1, j_1}, \dots, y_{i_k,j_k} \}$ which satisfy that the subsets $\{y_{i_1,j_1}, y_{i_2,j_2}\}, \dots, \{y_{i_1,j_1}, \dots, y_{i_k,j_k} \}$ span connected subgraphs of $\overline \Lambda$. 

Since $c$ is non-trivial and since by assumption $\varphi$ is injective, it follows that
$$
\varphi(c)= [ y_{i_1,1}\cdots y_{i_1, r_{i_1}}, \dots, y_{i_k,1} \cdots y_{i_k, r_{i_k}}] \ne 1
$$
and so using commutator identities we have that 
$$
\varphi(c) = \prod\limits_{s=(j_1, \dots, j_k), 1\le j_s \le r_{i_s} } {[y_{i_1, j_1}, \dots, y_{i_k, j_k}]}^{g(s)} \ne 1
$$
for some $g(s) \in \GG(\Gamma)$. It follows that there exists at least one tuple $(j_1, \dots, j_k)$ for which $[y_{i_1,j_1}, \dots, y_{i_k, j_k}]$ is non-trivial. Hence, from Lemma \ref{claim1'}, it follows that the subsets $\{y_{i_1, j_1}, y_{i_1, j_2}\}$,$\dots$,$\{y_{i_1,j_1}, \dots, y_{i_1,j_k} \}$ span connected graphs in $\overline \Lambda$ and so $\mathcal C$ is non-empty.

We define the extended non-commutation graph of the set $\{y_{i_1,j_1}, \dots, y_{i_k, j_k}\}$ as follows. For each element of the set $y_{i_s,j_s}$, we introduce a vertex $v_{s}$, $s=1,\dots, k$, so there are $k$ vertices; and there is an edge $(v_{s},v_{s'})$ if and only if $[y_{i_s, j_s}, y_{i_{s'}, j_{s'}}] \ne 1$. Note that although $y_{i_s, j_s}$ may be equal to $y_{i_{s'}, j_{s'}}$ in the group for some $s \ne s'$, in the graph the elements $y_{i_s, j_s}$ and $y_{i_{s'}, j_{s'}}$ are associated to different vertices $v_s$ and $v_{s'}$ correspondingly. 

We now use induction on $k$ to show that the non-commutation graph of the set $\{y_{i_1,j_1}, \dots, y_{i_k, j_k}\}$ is isomorphic to $T$. If $k=2$, since by assumption we have that $[y_{i_1, j_1}, y_{i_1, j_2}]\ne 1$, it follows that the non-commutation graph of $\{y_{i_1,j_1}, y_{i_2, j_2}\}$ is a path of length 1 and so isomorphic to the graph induced by $\{x_{i_1}, x_{i_2} \}$. 

Assume by induction that the non-commutation graphs defined by $\{y_{i_1, j_1}, \dots, y_{i_{k-1}, j_{k-1}}\}$ and $\{x_{i_1}, \dots, x_{i_{k-1}}\}$ are isomorphic. Since $T$ is a tree and the graph defined by $\{x_{i_1}, \dots, x_{i_{k-1}} \}$ is connected, it follows that $x_{i_k}$ is connected  to at most one vertex, say $x_{i_{s_0}}$. So, in particular, $x_{i_k}$ commutes with all the vertices from $\{x_{i_1}, \dots, x_{i_{k-1}} \} \setminus \{x_{i_{s_0}} \}$. It follows that $y_{i_k,1}, \dots, y_{i_k, r_{i_k}}$ pair-wise commute with $y_{i_t, 1}, \dots, y_{i_t, r_{i_t}}$, $t\in \{1, \dots, k-1\} \setminus \{s_0\}$, see Remark \ref{rem:edges}. Since the non-commutation graph of $\{y_{i_1, j_1}, \dots, y_{i_k, j_k}\}$ is connected, it follows that $y_{i_k, j_k}$ must be connected to $y_{i_{s_0}, j_{s_0}}$ and so the graphs defined by $\{y_{i_1,j_1}, \dots, y_{i_k, j_k}\}$ and $\{x_{i_1}, \dots, x_{i_k}\}$ are isomorphic.

\smallskip

A tuple $(y_{i_1,j_1}, \dots, y_{i_k,j_k})$ is called \emph{full} if and only if the commutator $[y_{i_1,j_1}, \dots, y_{i_k,j_k}]$ is non-trivial. We now show that there exist a  full tuple $(y_{i_1,j_1}, \dots, y_{i_k,j_k})$  whose elements are pair-wise distinct. Then, from the above discussion, it follows that the non-commutation graph $\Upsilon$ of such a tuple is an induced subgraph of $\overline \Lambda$; since $T$ is isomorphic to $\Upsilon$, we conclude that $T < \overline \Lambda$.

Assume towards contradiction, that each full tuple $(y_{i_1,j_1}, \dots, y_{i_k,j_k})$ has a repetition of elements, that is $y_{i_s,j_s} = y_{i_{s'},j_{s'}}$, for some $s\ne s'$. Let $(y_{i_1,j_1}, \dots, y_{i_k,j_k})$ be a full tuple with the maximal number of pair-wise distinct elements. If $y_{i_l,j_l} = y_{i_k,j_k}$, then $y_{i_l,j_l}$ and $y_{i_k,j_k}$ are connected by edges with the same elements in the non-commutation graph. Furthermore, since the non-commutation graph of the tuple is isomorphic to the tree $T$ and so has no cycles, it follows that $y_{i_l,j_l}$ and $y_{i_k,j_k}$ (and so $x_{i_l}$ and $x_{i_k}$) are leaves of the graph (that is have valency 1) and  are both connected by an edge to the same element from the tuple $(y_{i_1,j_1}, \dots, y_{i_k,j_k})$, say $y_{i_t,j_t}$. 

Note that if there is another $y_{i_l, m_l}$ in the image of $x_{i_l}$ connected to $y_{i_t,j_t}$ by an edge, then $y_{i_l, m_l}$ is equal to an element $y_{i_s,j_s}$ from the tuple, $s\ne l, k$. Indeed, otherwise, the tuple obtained by replacing $y_{i_l,j_l}$ by $y_{i_l,m_l}$, that is $(y_{i_1,j_1}, \dots, y_{i_l, m_l}, \dots, y_{i_k,j_k})$ is a full tuple (we only replace a leaf by a vertex connected to the tree) and has less repetitions of elements contradicting the choice of the tuple $(y_{i_1,j_1}, \dots, y_{i_k,j_k})$. Furthermore, arguing as above, since $y_{i_s, j_s} = y_{i_m, j_m}$ and $y_{i_m,j_m}$ is connected to $y_{i_t,j_t}$, it follows that $y_{i_s,j_s}$ is also connected to $y_{i_t,j_t}$ (and so $x_{i_s}$ is connected to $x_{i_t}$). Since $T$ is a tree and $y_{i_s, j_s}=y_{i_m,j_m}$, it follows that $y_{i_s,j_s}$ (and so $x_{i_s}$) is a leaf of the tree. 

Similarly, if there exists $y_{i_s, m_s}$ connected to $y_{i_t, j_t}$ by an edge and so that it does not appear in the tuple $(y_{i_1,j_1}, \dots, y_{i_k,j_k})$, then $(y_{i_1,j_1}, \dots, y_{i_k,j_k}) \setminus \{y_{i_l,j_l}, y_{i_s, j_s}\} \cup \{y_{i_l, m_l}, y_{i_s, m_s}\}$ is a full tuple with less repetitions, contradicting the choice of the tuple $(y_{i_1,j_1}, \dots, y_{i_k,j_k})$.

Repeating this argument we get a maximal set $\mathcal A= \{ x_{i_{l_1}}, \dots, x_{i_{l_A}} \}$ of vertices of $T$ such that each $x_{i_{l_a}}$ in $\mathcal A$ is a leaf of the graph $T$ connected to the vertex $x_{i_t}$ and such that each $y_{i_{l_a}, m_{i_{l_1}}}$ connected to $y_{i_t, j_t}$ appears in the tuple $(y_{i_1,j_1}, \dots, y_{i_k,j_k})$.

Consider $\varphi(x_{i_{l_a}})=y_{i_{l_a},1} \cdots y_{i_{l_a}, r_{l_a}} y_{i_{l_a},1}' \cdots y_{i_{l_a}, r_{l_a}'}'$, $a=1, \dots, A$, where $y_{i_{l_a}, f}$ is connected to $y_{i_t,j_t}$ and $y_{i_{l_a}, f'}'$ is not connected to $y_{i_t,j_t}$, $f=1, \dots, r_{l_a}$, $f'=1, \dots, r_{l_a}'$. By construction, $y_{i_{l_a},f}$ and $y_{i_{l_a},f'}'$ are pairwise distinct. On the other hand, it follows from the above argument that each $y_{i_{l_a},f}$, $f=1, \dots, r_{l_a}$ appears in the tuple $(y_{i_{l_1}, j_{l_1}}, \dots, y_{i_{l_A},j_{l_1}})$ and so $r_{l_a}$ is bounded by the number $r$ of different elements in the tuple $(y_{i_{l_1}, j_{l_1}}, \dots, y_{i_{l_A},j_{l_1}})$. In turn, since by assumption there is at least one repetition in the tuple $(y_{i_{l_1}, j_{l_1}}, \dots, y_{i_{l_A},j_{l_1}})$, we have that $r_{l_a} \le r < A$ for all $1\le a\le A$.

Recall that since $x_{i_{l_a}}$ are leaves of the non-commutation graph $T$, they generate a free abelian group. By Remark \ref{rem:edges}, the set of elements $y_{i_{l_a}, f}$, $f=1, \dots, r_{l_a}$, also generate a free abelian group. 

We order arbitrarily the set of $r$ pair-wise distinct elements of the set $\{ y_{i_{l_1}, j_{l_1}}, \dots, y_{i_{l_A}, j_{l_A}} \}$. To each $\varphi(x_{i_{l_a}})$, we associate a vector in $\BZ_2^r$ as follows: there is $1$ in the position $y_{i_{l_b}, j_{l_b}}$, $b=1, \dots, r$, if and only if $y_{i_{l_b}, j_{l_b}} = y_{i_{l_a}, f}$, for some $f=1, \dots r_{l_a}$, that is there is 1 in the position $y_{i_{l_b}, j_{l_b}}$ if and only if $y_{i_{l_b}, j_{l_b}}$ appears in the image $\varphi(x_{i_{l_a}})$ and is connected to $y_{i_t,j_t}$ by an edge. Then $\varphi(x_{i_{l_a}})$, $a=1, \dots, A$ determine $A$ vectors from $\BZ_2^r$ and since $r<A$ it follows that these vectors are linearly dependent.

Note that we do not claim that the images of $x_{i_{l_a}}$, $a=1, \dots, A$ are linearly dependent. The linear dependence involves only the elements that occur in the image and are connected to the vertex $y_{i_t,j_t}$; more precisely, there exist $\lambda_1, \dots, \lambda_A \in \BZ$ such that $\varphi( x_{i_{l_1}}^{\lambda_1} \cdots x_{i_{l_A}}^{\lambda_A})$ is a word in $y_{i_{l_a}, f'}'$, $a=1, \dots, A$, $f'=1, \dots, r_{l_a}'$, where by definition $y_{i_{l_a}, f'}'$ are not connected to $y_{i_t,j_t}$. Since $x_{i_{l_a}}$ are leaves of the non-commutation graph $T$, they are only connected to $x_{i_t}$ (and so commute with the other vertices). By Remark \ref{rem:edges}, it follows that if $b\ne t$, then $y_{i_{l_a},f'}'$ commutes with $y_{i_{l_b}, j_{l_b}}$. Furthermore, by definition $y_{i_{l_a}, f'}'$ are not connected to $y_{i_t, j_t}$ and so they commute with $y_{i_t,j_t}$. We conclude that there exist $\lambda_1, \ldots, \lambda_A$ such that 

\begin{equation}\label{eq:ld}
[y_{i_1,j_1}, \dots, y_{i_t, j_t},  \varphi(x_{i_{l_1}}^{\lambda_1} \cdots x_{i_{l_A}}^{\lambda_A})]=1. 
\end{equation}

In fact, for any other full tuple $(y_{i_1,j_1'}, \dots, y_{i_t, j_t}, \dots, y_{i_{l_A}, j_{l_A}'})$ that contains $y_{i_t, j_t}$, (not necessarily with minimal number of repetitions), we have that
\begin{equation}\label{eq:commutatortrivial}
[y_{i_1,j_1'}, \dots, y_{i_t, j_t},  \varphi(x_{i_{l_1}}^{\lambda_1} \cdots x_{i_{l_A}}^{\lambda_A})]=1.
\end{equation} 

Since $x_{i_{l_a}}$ are leaves of the tree $T$, there exists an ordering on the vertices of $T$
$$
x_{i_1}, \dots, x_{i_t}, x_{i_{l_1}}, \dots, x_{i_{l_A}}
$$ 
so that the subsets of vertices 
$$
\{x_{i_1}, x_{i_2}\}, \dots, \{x_{i_1}, \dots, x_{i_t}, x_{i_{l_1}}, \dots, x_{i_{l_A}}\}
$$ 
span connected subgraphs. We can extend this order to all vertices of $T$. More precisely, fix an ordering on the vertices of $T$: 
$$
x_{i_{t_1}}< \dots< x_{i_{t_h}}< x_{i_{c_1,1}}< \dots< x_{i_{c_1,r_1}}< \dots< x_{i_{c_m, 1}}< \dots< x_{i_{c_m, r_m}},
$$
where $x_{i_{t_1}}, \dots, x_{i_{t_h}}$ are vertices of valence at least $2$ and $x_{i_{c_s,1}}, \dots, x_{i_{c_s,r_s}}$ are leaves of the tree joined to the same vertex of valence greater than or equal to $2$, $s=1, \dots, m$.

Let $F=(y_{i_1, j_1}, \dots, y_{i_k,j_k})$ be a full tuple with minimal number of repetitions and let $s\in \{1, \dots, m\}$. Set $w_F(x_{i_{c_s,1}}, \dots, x_{i_{c_s,r_s}})$ to be the linear combination such that $\varphi(w_F)$ is a word that commutes with $y_{i_1, j_1}, \dots, y_{i_t,j_t}$, see Equation (\ref{eq:ld}). Let 
$$
W(s) = \prod\limits_{F \in \mathcal F} w_F(x_{i_{c_s,1}}, \dots, x_{i_{c_s,r_s}})^{N_F},
$$ 
where $\mathcal F$ is the set of full tuples with minimal number of repetitions and the integers $N_F \in \BZ$ are chosen in such a way that all the words $w_F(x_{i_{c_s,1}}, \dots, x_{i_{c_s,r_s}})$ appear in $W(s)$ as subwords.

Consider the commutator
$$
C=[x_{i_{t_1}}, \dots, x_{i_{t_h}}, W(1), \dots, W(m)].
$$
We claim that the commutator $C$ is non-trivial in $\GG(\Delta)$. Recall that since $x_{i_{c_s, 1}}, \dots, x_{i_{c_s, r_s}}$, $s\in \{1, \dots,m\}$, are leaves of the non-commutation tree $T$, they generate a free abelian group. Then the map $\pi$ that sends $x_{i_{c_s, 1}}$ to $x_{i_{c_s, 1}}$ and $x_{i_{c_s, 2}}, \dots, x_{i_{c_s, r_s}}$ to $1$, $s=1, \dots, m$, induces a projection that sends $C$ to $[x_{i_{t_1}}, \dots, x_{i_{t_h}}, x_{i_{c_s,1}}^{n_1}, \dots, x_{i_{c_m,1}}^{n_m}]$, $n_s \in \BZ\setminus\{0\}$, $s=1,\dots,m$. By the choice of order on the set of vertices of the tree $T$, each set of vertices $\{x_{i_{t_1}}, x_{i_{t_2}}\},\dots, \{x_{i_{t_1}}, \dots, x_{i_{t_h}}, x_{i_{c_1,1}}, \dots, x_{i_{c_m,1}}\}$ spans a connected subgraph. Hence, by Lemma \ref{claim1'} it follows that $\pi(C)=[x_{i_{t_1}}, \dots, x_{i_{t_h}}, x_{i_{c_s,1}}^{n_1}, \dots, x_{i_{c_m,1}}^{n_m}]$ is non-trivial and so neither is $C$.

On the other hand, using commutator identities and the fact that $[ y_{i_1,j_1'}, \dots, y_{i_t, j_t},  \varphi(W_F)]=1$, see Equation (\ref{eq:commutatortrivial}), we see that $\varphi(C)$ is trivial - a contradiction with the fact that $\varphi$ is injective. Hence there is a full tuple $(y_{i_1,j_1}, \dots, y_{i_k,j_k})$ without repetitions of elements and so the graph $T$ is an induced subgraph of $\overline\Lambda$.

\medskip

Finally, we are left to show that if every connected component $T$ of $\overline \Delta$ embeds into $\overline \Lambda$, then 
$\overline \Delta$ also embeds into $\overline \Lambda$. We prove it by induction on the number of connected components. Assume that the forest $\overline \Delta$ is the disjoint union of $n$ trees $T_1', \dots, T_n'$. Let $T_1 = T_1'$ and $T_2$ be the disjoint union of $T_2', \dots, T_n'$. Note that since $T_1$ and $T_2$ are disjoint in the non-commutation graph $\overline \Delta$, it follows that $\GG(T_1)$ and $\GG(T_2)$ commute. Recall that $\varphi: \GG(\Delta) \to \GG(\Gamma)$ is an embedding defined by the map $x_i \to y_{i1} \cdots y_{ir_i}$, $i=1, \dots, n$ and by Remark \ref{rem:edges} there are no edges between $y_{ik}$ and $y_{jl}$ for all $x_i \in T_1$ and $x_j \in T_2$. By induction we assume that there exist injective maps $\psi_l: T_l \to \overline \Lambda$, $l=1,2$ such that $\psi_l(x_i) = y_{ij_{il}}$, where $j_{il} \in \{1,\dots, r_i\}$, $l=1,2$. Let us show that there exists $\psi:\overline \Delta \to \overline \Lambda$ such that $\psi(x_i)=y_{ij_i}$, where $j_{i} \in \{1,\dots, r_i\}$. If $\psi_1(T_1)$ and $\psi_2(T_2)$ are disjoint, then we can define $\psi(T_1)=\psi_1(T_1)$ and $\psi(T_2)=\psi_2(T_2)$ and clearly $\psi$ satisfies the induction hypothesis.

Assume that $T_1$ is just one vertex, say $x_1$, and that $\psi_1(T_1)$ and $\psi_2(T_2)$ are not disjoint, say $\psi_1(x_1) = y_{1j_{11}} = y_{ij_{i2}} = \psi_2(x_i)$, for some $x_i$ in $T_2$. Since $x_1$ and $x_i$ generate a free abelian subgroup of rank 2 in $\GG(\Delta)$ and since $\varphi$ is injective, we have that $\varphi(x_1)$ and $\varphi(x_i)$ generate a free abelian group of rank 2. It follows that either $\varphi(x_1)=y_{1,1} \cdots y_{1, {r_1}}$ and $r_1 >1$ or $\varphi(x_i)=y_{i_1, 1}\cdots y_{i_i,r_i}$ and $r_i>1$. If $r_1>1$, define $\psi(T_2) = \psi_2(T_2)$ and $\psi(T_1)=\psi(x_1)= y_{1j}$, where $y_{1j} \ne y_{1j_{11}}$, $j\in \{1, \dots, r_1\}$. Similarly, if $r_i>1$, define $\psi(T_2) = \psi_2(T_2)$ and $\psi(T_1)=\psi(x_1)= y_{ij}$, where $y_{ij} \ne y_{ij_{i2}}$, $j\in \{1, \dots, r_i\}$. Then $\psi$ satisfies the induction hypothesis.

On the other hand, if $T_1$ is not a vertex, then we claim that $\psi_1(T_1)$ and $\psi_2(T_2)$ are disjoint. Indeed, assume towards contradiction that $\psi_1(x_1)=y_{1j_{11}}=y_{2j_{22}}=\psi_2(x_2)$, for some $x_1 \in T_1$ and $x_2 \in T_2$. Let $x_3$ in $T_1$ be connected to $x_1$, that is $(x_1, x_3)\in E(\overline \Delta)$. Since $T_1$ and $T_2$ are disjoint, it follows that $x_3 \in T_1$ commutes with $x_2 \in T_2$. From Remark \ref{rem:edges}, we have that if $\varphi(x_3)=y_{31} \cdots y_{3r_3}$ and $\varphi(x_2)=y_{21} \cdots y_{2r_2}$, then $(y_{2l_2}, y_{3l_3}) \notin E(\overline \Lambda)$, $l_s=1, \dots, r_s$, $s=2,3$. It follows that $\psi_1(x_3)=y_{3,j_{31}}$ is not connected to $\psi_1(x_1)=y_{1j_{11}}=y_{2j_{22}}$, that is $(\psi_1(x_1), \psi_1(x_3)) \notin E(\overline \Lambda)$ -a contradiction since $(x_1,x_3) \in E(\overline \Delta)$ and $\psi_1$ is injective. This shows that there exists $\psi: \overline \Delta \to \overline \Lambda$ satisfying the induction hypothesis and so $\overline \Delta$ is an induced subgraph of $\overline\Lambda$.
\end{proof}

\bigskip

Combining the main results of this section, we derive the following corollaries.

\begin{cor}\label{cor:decemb2dim}
There exists an algorithm that given a simplicial graph $\Delta$ and a triangle-free or triangle-built graph $\Gamma$ decides whether or not there exists an embedding from $\GG(\Delta)$ to $\GG(\Gamma)$. In other words, it is decidable whether or not an arbitrary pc group is a subgroup of a triangle-free or triangle-built pc group.
\end{cor} 
 
\begin{cor} \label{cor:comt}
There is an algorithm that given a forest or a complement of a forest $\Delta$ and an arbitrary simplicial graph $\Gamma$ decides whether or not there exists an embedding from $\GG(\Delta)$ to $\GG(\Gamma)$. In other words, it is decidable whether or not a pc group defined by a forest or a complement of a forest is a subgroup of an arbitrary pc group $\GG(\Gamma)$.
\end{cor}

\begin{cor}
There exists an algorithm that given a simplicial graph $\Delta$ and a graph $\Gamma$ whose deflation is triangle-free decides whether or not there exists a tame embedding from $\GG(\Delta)$ to $\GG(\Gamma)$.
\end{cor}
\begin{proof}
It follows from Theorems \ref{thm:KKtfree}, \ref{thm:CDKtbuilt}, \ref{thm:defltfree}, \ref{thm:KKtree} and \ref{thm:noncommtree} that the corresponding Embedding Problems are equivalent to the Extension Graph Embedding Problem which is decidable by Corollary \ref{cor:decEGE}.
\end{proof}

\section{Embeddability and universal equivalence}\label{sec:universal}

The goal of this section is to establish a relation between the model-theoretic problem of classification of pc groups up to universal equivalence and the algebraic problem of embeddability between pc groups. More concretely, we establish when two pc groups are universally equivalent in terms of mutual embeddability into the inflation pc groups.

\begin{thm11}
The partially commutative groups $\GG(\Delta)$ and $\GG(\Gamma)$ are universally equivalent if and only if there exist tame embeddings $\varphi:\GG(\Delta) \to \GG(\Gamma_n)$ and $\psi:\GG(\Gamma) \to \GG(\Delta_m)$, where $n$ and $m$ are the number of {\rm(}canonical{\rm)} generators of $\GG(\Delta)$ and $\GG(\Gamma)$ respectively. 
\end{thm11}

Recall that the $n$-inflation of a partially commutative group $\GG(\Delta_n)$ is the graph product with underlying graph $\Delta$ and $n$-generated free abelian vertex groups.

\medskip

We use this characterisation to transfer algebraic and algorithmic results on embeddability to model-theoretic ones. For instance, we deduce that two pc groups defined by cycles are elementarily equivalent if and only if they are isomorphic; similarly, we use the fact that the (Tame) Embedding Problem is decidable when the inflation graph of the target group is triangle-free, to prove that there is an algorithm to decide when a pc group is universally equivalent to a given $2$-dimensional pc group.

\medskip

We begin by proving that a pc group and its $n$-inflation are universally equivalent. This result is a natural generalisation of the fact that free abelian groups are universally equivalent.

\begin{lemma}\label{lem:disextcen}
Let $\GG$ be a pc group and let $\HH$ be a (free) extension of a centraliser:
$$
\HH=\langle \GG, A \mid [A,C_\GG(x)]=1 \rangle,
$$
where $A \simeq \BZ^{n-1}$ and $x$ is a generator of $\GG$. Then $\HH$ is discriminated by $\GG$. 
\end{lemma}

Recall that a family of homomorphisms $\varphi_n: H \to G$ is called \emph{discriminating} if for any finite set $S$ of elements in $H$, there exists a homomorphism $\varphi_{n_S}$ in the family which is injective in the set $S$. We say that $H$ is \emph{discriminated} by $G$ if there exists a discriminating family.

\begin{proof}
Since $\HH$ is an HNN-extension, any non-trivial element $h$ of $\HH$ can be written as follows:
$$
h=g_1 u_1 g_2 u_2 \dots g_k u_k g_{k+1}
$$
where $g_i \in \GG$, $i=1, \dots, k+1$, $u_j \in A$, $u_j \ne 1$, $j=1, \dots, k$, $g_i \notin C_\GG(x)$, $i=2, \dots, k$.
The free abelian group $A$ is discriminated by $\langle x \rangle$. Let $\phi_r$ be a discriminating family for $A$. Any homomorphism $\phi_r : A \to \langle x \rangle$ extends naturally to a homomorphism $\phi_r': \HH \to \GG$ whose restriction onto $\GG$ is the identity.  Let $c_0$ be the maximal exponent of $x$ in the coefficients $g_i$, $i=1, \dots, k+1$, more precisely, for every geodesic word representing $g_i$ of the form $w_1 x^c w_2$, where $x^{\pm 1}$ does not left-divide $w_1$ and does not right-divide $w_2$ one has that $|c| \le c_0$. For the infinite subfamily $\{\phi_r\}$ so that $\phi_r(u_j) =x^{s_k}$, where $|s_k| >2 |c_0|$ for all $j=1, \dots, k$, we have that $\phi_r' (h)=g_1' x^{t_1} \dots g_k' x^{t_k} g_{k+1}'$, $|t_i| \ge |s_i| - 2|c_0| > 1$, $g_i' \notin C(x)$. Indeed, to see this it suffices to notice that since $g_i \notin C_\GG(x)$, so $x^{\epsilon c_0} g_i x^{\delta c_0}$, $\epsilon, \delta \in \{\pm 1\}$ is non-trivial and it does not belong to $C_\GG(x)$. Hence, if we view $\GG$ as the HNN-extension $\GG=\langle \GG_{n-1}, x \mid (C(x) \cap \GG_{n-1})^x=(C(x) \cap \GG_{n-1}) \rangle$, where $\GG_{n-1}$ is the partially commutative group generated by all generators of $\GG$ but $x$, then $\phi_r'(h)$ is a reduced element of $\GG$. It follows that $\phi_r'(h)$ is non-trivial for an infinite family of homomorphisms $\phi_r'$. Therefore the family $\phi_r'$ discriminates $\HH$ to $\GG$.
\end{proof}

\begin{prop}\label{prop:influniveq}
The partially commutative group $\GG(\Delta)$ and its $m$-inflation $\GG(\Delta_m)$ are universally equivalent, for all $m\in \mathbb{N}$.
\end{prop}
\begin{proof}
Since partially commutative groups are linear and since $\GG(\Delta) < \GG(\Delta_m)$, in order to show that $\GG(\Delta)$ and $\GG(\Delta_m)$ are universally equivalent it suffices to prove that $\GG(\Delta_m)$ is discriminated by $\GG(\Delta)$, see \cite{DMR}. 

Observe that the $m$-inflation group $\GG(\Delta_m)$ can be constructed as an iterated sequence of extension of centralisers of the generators $x_i$, $i=1, \dots,n$ of the partially commutative group $\GG(\Delta)$, that is
\begin{gather}
\begin{split}
\GG(\Delta) & = \GG_0< \\
& < \GG_1=\langle \GG(\Delta), y_{12}, \dots, y_{1m} \mid [y_{1j}, C_{\GG_0}(x_1)]=1, [y_{1i},y_{1j}]=1, 1\le i<j \le m \rangle < \\
& < \GG_2<\dots <\GG_{n-1}< \\
& < \GG_n= \left< \GG_{n-1}, y_{n2}, \dots, y_{nm} \left| 
\begin{array}{l}
 [y_{nj}, C_{\GG_{n-1}}(x_n)]=1, \\
 \left[y_{ni},y_{nj}\right]=1, 1\le i<j \le m
\end{array}
 \right.\right> = \GG(\Delta_m).
\end{split}
\end{gather}

\smallskip

Since a free extension of a centraliser of a generator of a partially commutative group is again a partially commutative group, the statement follows from Lemma \ref{lem:disextcen}.
\end{proof}

\begin{cor}\label{cor:onlyif}
If there exist embeddings $\varphi:\GG(\Delta) \to \GG(\Gamma_n)$ and $\psi:\GG(\Gamma) \to \GG(\Delta_m)$ for some $n,m\in \BN$, then $\GG(\Delta) \equiv_\forall \GG(\Gamma)$.
\end{cor}

\begin{proof}
If $\GG(\Delta)$ embeds into $\GG(\Gamma_n)$, then $\Th_\forall(\GG(\Gamma_n)) \subset \Th_\forall(\GG(\Delta))$ and so it follows from Proposition \ref{prop:influniveq} that  $\Th_\forall(\GG(\Gamma)) = \Th_\forall(\GG(\Gamma_n)) \subset \Th_\forall(\GG(\Delta))$. Similarly, if $\GG(\Gamma) < \GG(\Delta_m)$, then $\Th_\forall(\GG(\Delta)) \subset \Th_\forall(\GG(\Gamma))$ and so the statement follows.
\end{proof}

\medskip

We now prove that if the existential theory of $\GG(\Delta)$ is contained in the existential theory of $\GG(\Gamma)$, then there exists a tame embedding from $\GG(\Delta)$ to the $n$-inflation $\GG(\Gamma_n)$. The proof relies on the following intermediate result on embeddability, which we think is interesting on its own right. We show that in order to determine if a tame map between pc groups is an embedding, it suffices to check if it is injective on a finite set of elements. More precisely, we prove in Lemma \ref{cor:embfinitecond} that given two simplicial graphs $\Delta$ and $\Gamma$, $|V(\Delta)|=n$, there exists a finite set of words $W \subset F(x_1, \dots, x_n)$ such that if a tame map $f: \GG(\Delta) \to \GG(\Gamma)$ is injective on the set $W$, then $f$ is an embedding.

The strategy of the proof is as follows. One can encode the existence of a map which is injective on the finite set of words $W$ in an existential sentence $\varphi_{\Delta,W}$, i.e. if a tuple $\overline g$ from $\GG(\Gamma)$ witnesses the existential sentence $\varphi_{\Delta, W}$, then $W(g)=1$ in $\GG(\Gamma)$ if and only if $W=1$ in $\GG(\Delta)$. The group $\GG(\Delta)$ satisfies the sentence $\varphi_{\Delta, W}$, say in the tuple of generators. Therefore, if $\Th_\exists(\GG(\Delta)) \subset \Th_\exists(\GG(\Gamma))$, it follows that $\GG(\Gamma) \models \varphi_{\Delta,W}(\overline g)$, for some tuple $\overline g$ and so the map $\overline x \to \overline g$ is injective on the set $W$. To assure that this map is tame, we inflate the graph $\Gamma$; finally, we use the intermediate result that tame maps which are injective on the set $W$ are embeddings to deduce the existence of an embedding from $\GG(\Delta)$ to $\GG(\Gamma_n)$.

The outline of the proof of the intermediate embedding result is as follows. We first show that to any $n$-tuple of elements $\overline g$ from $\GG(\Gamma)$, one can associate a finite ($n$-coloured) subgraph $\Lambda_{\overline g}$ of $\Gamma_n^e$ whose number of vertices is bounded by $\Delta$ and $\Gamma$. Thus, one can define a map from the set $\GG(\Gamma)^n$ to a finite set of finite subgraphs of $\Gamma_n^e$. In turn, to a (coloured) graph $\Lambda_{\overline g}$, we associate an $n$-tuple $\overline a$ of elements from $\GG(\Gamma_n)$. A crucial point is that if the tuple of elements $\overline g$ witnesses an existential sentence $\varphi$ in $\GG(\Gamma)$, that is $\GG(\Gamma) \models \varphi(\overline g)$,  then so does the $n$-tuple $\overline a$, i.e. $\GG(\Lambda_{\overline g})\models \varphi(\overline a)$; see Lemma \ref{lem:elemgraph}. The fact that the number of graphs $\Lambda_{\overline g}$ is finite allows us to prove that there is a Noetherian-type property for the set of existential sentences:
$$
\varphi_{\Delta, m}= \exists x_1, \dots, x_n \left( \bigwedge\limits_{(x_i,x_j)\in E(\Delta)} [x_i,x_j]=1 \wedge \bigwedge\limits_{i=1, \dots, m} w_i \ne 1\right),
$$
where $m\in \BN$ and $w_i$ belongs to the infinite set of words $\mathcal W$ that represent non-trivial elements of $\GG(\Delta)$, i.e. 
$$
w_i \in \mathcal W=\{ w_i \in F(x_1, \dots, x_n) \mid w_i \ne 1 \hbox{ in } \GG(\Delta)\}.
$$

In other words, we show that there exist finitely many words $w_1, \dots, w_s \in \mathcal W$ such that if there is a tuple $\overline a$ that witnesses the existencial sentence defined by these finitely many words, i.e. there exists $\overline a$ such that  $\GG(\Lambda_{\overline g}) \models \varphi_{\Delta,s}(\overline a)$ for some $\Lambda_{\overline g}$, then the tuple $\overline a$ witnesses every existencial sentence of the  infinite family, that is  $\GG(\Lambda_{\overline g}) \models \varphi_{\Delta, m}(\overline a)$, for all $m\in \mathcal \BN$. In terms of embeddings, the latter statement can be read as follows: if the map from $\GG(\Delta)$ to $\GG(\Lambda_{\overline{g}})$ that sends the tuple of generators $\overline x$ of $\GG(\Delta)$ to $\overline a$ is injective on the set of words $w_1, \dots, w_s$, then that map is injective on the set of \emph{all} non-trivial words of $\GG(\Delta)$ and so it is an embedding. Let us formalise these ideas.

\smallskip

A \emph{coloured graph} is a graph $\Gamma$ together with a map from the set of vertices $V(\Delta)$ to a finite set $S$. Elements of the set $S$ are called \emph{colours} and the map, a \emph{colouring}. 

We now turn our attention to the construction of coloured graphs associated to tuples of elements from $\GG(\Gamma)$. Informally, given an $n$-tuple of elements $\overline g$ from $\GG(\Gamma)$, one can consider the set of blocks that appear in the decomposition of the elements of the tuple. Note that a block can appear in the decomposition of two different elements. In order to avoid this (and thinking towards obtaining a tame map), we consider the $n$-inflation of $\Gamma$ and replace the blocks that appear in decompositions of different elements by the same word in new generators; doing so, we end up with at most $dn$ different blocks (where $d$ is the maximal size of a clique in $\Gamma$). Taking powers of these  blocks, we get a pc group with underlying graph $\Lambda_{\overline g}$, see \cite{Koberda}. We associate $\Lambda_{\overline g}$ to the tuple $\overline g$. Formally,

\begin{defn}\label{def:graphasstog}
Let $V(\Gamma)=\{y_1, \dots, y_m\}$ and $V(\Gamma_n)=\{y_1, \dots, y_{1n}, \dots, y_m, \dots, y_{mn}\}$, where $C(y_{i,j})=C(y_i)$ in $\GG(\Gamma_n)$, $j=2, \dots,n$.

Let $\overline g=(g_1, \dots, g_n)$ be a tuple of elements in $\GG(\Gamma)$. Let $g_i= b_{i1} \cdots b_{ir_i}$ be the block decomposition of $g_i$, $i=1, \dots, n$. By \cite{Koberda}, large enough powers of block elements generate a partially commutative group with underlying graph $\Lambda'$, where $\Lambda'$ is an induced subgraph of $\Gamma^e$. Let $i:\Lambda' \to \Gamma^e$ be an embedding, $i(b_{ij})=y_{ij}^{w_{ij}}$, where $y_{ij}$ is a generator of $\GG(\Gamma)$.

Note that by definition of the block decomposition, it follows that $b_{ij} \ne b_{ik}$ if $i \ne k$. Hence, a block $b_{ij}$ repeats at most $n$ times in the list $\{b_{11}, \dots, b_{n r_n}\}$ and so $y_{ij}^{w_{ij}}$ repeats at most $n$ times in $\{ y_{11}^{w_{11}}, \dots, y_{nr_n}^{w_{nr_n}}\}$.

We replace the vertices $\{ y_{11}^{w_{11}}, \dots, y_{nr_n}^{w_{nr_n}}\}$ of $\Gamma^e$ by vertices of $\Gamma_n^e$ as follows. If 
$$
y_{i_1j_1}^{w_{i_1j_1}}=\dots=y_{i_rj_r}^{w_{i_rj_r}}
$$ 
are all the repetitions of the element $y_t^{w_{t}}$, $y_t \in V(\Gamma)$, then $r\le n$ and we replace $y_{i_sj_s}^{w_{i_sj_s}}$ by $y_{ts}^{w_{i_sj_s}}$, where $y_{ts} \in V(\Gamma_n)$, $s=2, \dots, r$.

By doing this replacement, we obtain a list of pairwise different vertices of $\Gamma_n^e$ that we denote by $\{a_{11}, \dots, a_{1r_1}, \dots, a_{n1}, \dots, a_{nr_n}\}$. Let $\Lambda_{\overline g}$ be the graph induced by this set in $\Gamma_n^e$. We colour the vertices $a_{ij}$, $j=1, \dots, r_i$ with the colour $i$ and so $\Lambda_{\overline g}$ is $n$-coloured.

Notice that by construction the number of vertices of $\Lambda_{\overline g}$ is bounded by $nd$, where $d= \max\{ |V(C)| \mid C \hbox{ is a clique in } \Gamma\}$  and $n=|V(\Delta)|$. Hence the bound $nd$ is independent of the tuple $\overline g$.

To a tuple $\overline g=(g_1, \dots, g_n)$ in $\GG(\Gamma)$, we associate the finite $n$-coloured graph $\Lambda_{\overline g} < \Gamma_n^e$ just constructed.
\end{defn}

We next show that if the tuple $\overline g$ witnesses a (particular type of) existential sentence, so does the tuple $\overline a= (a_{11} \cdots a_{1r_1}, \dots, a_{n1}\cdots,a_{nr_n})$ defined above.

Given a simplicial graph $\Delta$, $V(\Delta)=\{x_1, \dots,x_n\}$ and a finite set of words $W=\{w_1, \dots, w_k\}$, where $w_i \in F(x_1, \dots, x_n)$, $i=1, \dots, k$, we define the formula $\varphi_{\Delta, W}$ as follows:
$$
\varphi_{\Delta, W} = \exists x_1, \dots, x_n \left( \bigwedge\limits_{(x_i,x_j)\in E(\Delta)} [x_i,x_j]=1 \wedge \bigwedge\limits_{i=1, \dots, k} w_i \ne 1\right)
$$

\begin{lemma}\label{lem:elemgraph}
In the above notation, if 
$$
\GG(\Gamma) \models \varphi_{\Delta, W}(\overline g),
$$ 
then 
$$
\GG(\Lambda_{\overline g}) \models \varphi_{\Delta, W}(\overline a).
$$
\end{lemma}
\begin{proof}
In the notation of Definition \ref{def:graphasstog}, it follows from the construction of the graph $\Lambda_{\overline g}$ that the map $a_{ij}=y_{ts}^{w_{i_sj_s}} \to y_{t}^{w_{i_sj_s}}$ induces an epimorphism:
$$
\pi_1: \GG(\Lambda_{\overline g}) \to \GG(\Lambda').
$$

On the other hand, it was shown in \cite{Koberda} that for $N$ large enough the elements $\{b_{ij}^N\}$ generate a partially commutative group  and so the homomorphism $$
\pi_2: \langle b_{11}^N, \dots, b_{nr_n}^N \rangle \to \langle b_{11}, \dots, b_{nr_n} \rangle
$$ 
induced by the map $b_{ij}^N \to b_{ij}$ precomposed with the inverse of the isomorphism $i: \langle b_{11}^N, \dots, b_{nr_n}^N \rangle \to \GG(\Lambda')$ induces an epimorphism:
$$
\pi_2 i^{-1}: \GG(\Lambda') \to \langle b_{11}, \dots, b_{1r_1}, \dots, b_{n1}, \dots, b_{nr_n} \rangle < \GG(\Gamma)
$$
Furthermore, it follows from the definitions that 
\begin{equation}
\begin{split}
\pi_2 i^{-1} \pi_1(a_{i1} \cdots a_{ir_i}) & = \pi_2 i^{-1} (y_{i1}^{w_{i1}} \cdots y_{ir_i}^{w_{ir_i}}) = \\ & = \pi_2 ( b_{i1}^N \cdots b_{ir_i}^N) = b_{i1}\cdots b_{ir_i}= g_i,
\end{split}
\end{equation}
where $i=1, \dots, n$.

Assume that $\GG(\Gamma) \models \varphi_{\Delta, W} (\overline g)$, i.e 
\begin{itemize}
\item $[g_i, g_j]=1$, if $(x_i,x_j) \in E(\Delta)$ and
\item $w(\overline g)\ne 1$.
\end{itemize} 
We want to show that 
$$
\GG(\Lambda_{\overline g}) \models \varphi_{\Delta, W}(\overline{a_1}, \dots, \overline{a_n}),
$$
where $\overline{a_i}=a_{i1}\cdots a_{ir_i}$, $i=1, \dots, n$, i.e.
\begin{itemize}
\item $[\overline{a_i}, \overline{a_j}]=1$, if $(x_i,x_j) \in E(\Delta)$ and
\item $w(\overline{a_1}, \dots, \overline{a_n})\ne 1$.
\end{itemize} 

On the one hand, the above equalities hold since centralisers in pc groups are isolated. Indeed, if $[g_i, g_j]=1$, since $g_i=b_{i1}\cdots b_{ir_i}$, it follows that $[b_{i1} \cdots b_{ir_i}, b_{j1} \cdots b_{jr_j} ]=1$ and so $[b_{i1}^N \cdots b_{ir_i}^N, b_{j1}^N \cdots b_{jr_j}^N ]=1$. Using the isomorphism $i$ and the fact that $C(y_{ts})=C(y_t)$, we have that $[a_{i1} \cdots a_{ir_i}, a_{j1} \cdots a_{jr_j} ]=1$.

On the other hand, the existence of the morphism $\pi_2 \pi_1$ assures that the above inequalities hold. Indeed, $\pi_2(\pi_1(w(\overline{a_1}, \dots, \overline{a_n})))=w(g_1, \dots, g_n)$ and since $\GG(\Gamma) \models \varphi_{\Delta, W} (\overline g)$, we have that $w(g_1, \dots, g_n) \ne 1$. It follows that $w(\overline{a_1}, \dots, \overline{a_n})\ne 1$. 
\end{proof}

\begin{remark}
Notice that the witness of the formula $\varphi_{\Delta,W}$ in $\GG(\Lambda_{\overline g})$ only depends on the colouring of the graph $\Lambda_{\overline g}$ and not on $\overline g$, that is if $\GG(\Gamma) \models \varphi_{\Delta,W}(\overline g) \wedge \varphi_{\Delta,W'}(\overline h)$ and $\Lambda_{\overline g}=\Lambda_{\overline h}$ as coloured graphs, then $\GG(\Lambda_{\overline g})=\GG(\Lambda_{\overline h}) \models \varphi_{\Delta,W} (a_{11}\cdots a_{1r_1}, \dots, a_{n1}\cdots a_{nr_n}) \wedge \varphi_{\Delta,W'} (a_{11}\cdots a_{1r_1}, \dots, a_{n1}\cdots a_{nr_n})$.
\end{remark}

\begin{remark}
{\rm Lemma \ref{lem:elemgraph}} shows that inflation of the graphs is necessary. In general, if one considers the graph defined by the block elements {\rm(}without replacing repetitions{\rm)} the statement is trivially false: consider the tuple $\overline g= (a, a^2)$ in $\GG(\Gamma)$; then $\Lambda'$ is the graph with one vertex labeled by $A$ and $\Lambda_{\overline g}$ would be the graph with two vertices $A$ and $A'$ and one edge between them. Let $\varphi$ be the following existential sentence $\exists x,y  \ x\ne y$. Then $\GG(\Gamma) \models \varphi(\overline g)$, $\GG(\Lambda_{\overline g}) \models \varphi(A,A')$ but $\GG(\Lambda') \not\models \varphi(A,A)$.
\end{remark}

\begin{defn}
Given two simplicial graphs, $\Delta$ and $\Gamma$, $|V(\Delta)|=n$ and $w\in F(x_1, \dots, x_n)$, define $wit(w)$ to be the set of $n$-coloured graphs $\Lambda$ with at most $nd$ vertices for which there exist $\overline g \in \GG(\Gamma)^n$ such that $\GG(\Gamma) \models \varphi_{\Delta,w} (\overline g)$ and $\Lambda=\Lambda_{\overline g}$ (as coloured graphs), i.e.
$$
wit(w)=\{ \Lambda < \Gamma_n^e \mid |V(\Lambda)| \le nd, \Lambda=\Lambda_{\overline g}, \GG(\Gamma) \models \varphi_{\Delta,w} (\overline g) \},
$$
where $d$ is the maximal size of a clique in $\Gamma$.

Note that since the number of vertices of the graphs $\Lambda$ is bounded, the set $wit(w)$ is finite.
\end{defn}

\begin{defn}\label{def:witness}
Let $\Delta$ be a simplicial graph, $|V(\Delta)|=n$ and let $\mathcal W \subset F(x_1, \dots, x_n)$ be a (possibly infinite) set of words. A set of $i$-\emph{witnesses} for $\Delta$, $\Gamma$ and $\mathcal W$ is a finite set of words $w_1, \dots, w_s \in \mathcal W$ such that
\begin{enumerate}
\item $wit(w_i) \ne \emptyset$, $i=1, \dots, s$;
\item $ \emptyset \subseteq \bigcap\limits_{i=1, \dots, s} wit(w_i) \subsetneq \bigcap\limits_{i=1, \dots, s-1} wit(w_i) \subsetneq \dots \subsetneq \bigcap\limits_{i=1,2} wit(w_i) \subsetneq wit(w_1);$
\item for any $w\in \mathcal W$ such that $wit(w) \ne \emptyset$, we have that 
$$
\bigcap\limits_{i=1, \dots, s} wit(w_i) \cap wit(w) = \bigcap\limits_{i=1, \dots, s} wit(w_i).
$$
\end{enumerate}

Note that if $\Gamma$ is non-empty, then the set of $i$-witnesses always exists. Indeed, it suffices to notice that there are words  $w$ so that $wit(w)\neq \emptyset$, for instance, the word $w=x_1$. Then the set of $i$-witnesses exists since the set $wit(w)$ is finite, for every word $w$.
\end{defn}

\begin{lemma}\label{lem:i-witness}
Let $\Delta$ and $\Gamma$ be simplicial graphs, $|V(\Delta)|=n$. Let $\mathcal W < F(x_1, \dots, x_n)$ be the set of words such that $w\ne 1$ in $\GG(\Delta)$. If $\Th_\exists(\GG(\Delta)) \subset \Th_\exists(\GG(\Gamma))$, then there exists a set of $i$-witnesses $W=\{w_1, \dots, w_s\}$ such that $\bigcap\limits_{i=1, \dots, s} wit(w_i)$ is non-empty.
\end{lemma}
\begin{proof}
Let $V(\Delta)=\{X_1, \dots, X_n\}$. It follows from the definition of $\mathcal W$ that for any finite set of words $W \subset \mathcal W$, we have $\GG(\Delta) \models \varphi_{\Delta, W}(X_1, \dots, X_n)$. Since $\Th_\exists(\GG(\Delta)) \subset \Th_\exists(\GG(\Gamma))$ and $\varphi_{\Delta, W}$ is an existential sentence, it follows that there exists $w\in \mathcal W$ for which $wit(w) \ne \emptyset$ and so the set of $i$-witnesses $\{w_1, \dots, w_s\}$ exists. Then, if we set $W$ to be the set of $i$-witnesses $\{w_1, \dots, w_s\}$, we have that $\GG(\Gamma) \models \varphi_{\Delta, W}(\overline g)$ and so $\Lambda_{\overline g} \in \bigcap\limits_{i=1, \dots, s} wit(w_i)$.
\end{proof}

\begin{cor}\label{cor:i-witness}
In the notation of {\rm Lemma \ref{lem:i-witness}}, for any coloured graph $\Lambda \in \bigcap\limits_{i=1, \dots, s} wit(w_i)$, we have that $\GG(\Lambda) \models \varphi_{\Delta, w}(a_{11} \cdots a_{1r_1}, \dots, a_{n1} \cdots a_{nr_n})$, for all $w\in \mathcal W$.
\end{cor}
\begin{proof}
The statement is an immediate consequence of the fact that the chain of witnesses is maximal, see Condition (3) in Definition \ref{def:witness}. Indeed, since by definition, for all $w\in \mathcal W$ we have that 
$$
\bigcap\limits_{i=1, \dots, s} wit(w_i) \cap wit(w) = \bigcap\limits_{i=1, \dots, s} wit(w_i)
$$ 
and since $\Lambda \in \bigcap\limits_{i=1, \dots, s} wit(w_i)$, we have that $\Lambda \in wit(w)$.
\end{proof}

\begin{thm}\label{thm:univeralimpliesembedd}
If $\Th_\exists(\GG(\Delta)) \subset \Th_\exists(\GG(\Gamma))$, then there exists a tame embedding $\varphi: \GG(\Delta) \to \GG(\Gamma_n)$.
\end{thm}
\begin{proof}
If $\Th_\exists(\GG(\Delta)) \subset \Th_\exists(\GG(\Gamma))$, then, by Lemma \ref{lem:i-witness}, there exists $\Lambda \in wit(w)$ for all $w\in \mathcal W$, where $\mathcal W$ is the set of words that represent non-trivial elements in $\GG(\Delta)$. Furthermore, since by Corollary \ref{cor:i-witness}, $\GG(\Lambda) \models \varphi_{\Delta, w}(a_{11} \cdots a_{1r_1}, \dots, a_{n1} \cdots a_{nr_n})$, for all $w\in \mathcal W$, it follows that the map $X_i \to a_{i1} \cdots a_{ir_i}$ induces an embedding from $\GG(\Delta)$ to $\GG(\Lambda)$. By construction of $\Lambda$, we have that $a_{ij} \ne a_{kl}$ if $i\ne k$ or $j \ne l$ and so the embedding is tame. By definition, $\Lambda < \Gamma_n^e$ and so by \cite{KK} it follows that there is a (tame) embedding of $\GG(\Lambda)$ into $\GG(\Gamma_n)$.
\end{proof}

\begin{thm}\label{cor:univequivstronemb}
Let $\Delta$ and $\Gamma$ be two simplicial graphs, $|V(\Delta)|=n$ and $|V(\Gamma)|=m$. Then, $\GG(\Delta) \equiv_\forall \GG(\Gamma)$ if and only  if there exist tame embeddings $\varphi:\GG(\Delta) \to \GG(\Gamma_n)$ and $\psi: \GG(\Gamma) \to \GG(\Delta_m)$.
\end{thm}
\begin{proof}
The statement is an immediate consequence  of Theorem \ref{thm:univeralimpliesembedd} and Corollary \ref{cor:onlyif}.
\end{proof}

The above results can be interpreted algebraically as follows. We have proven that there exists a tame embedding from $\GG(\Delta)$ to $\GG(\Gamma)$ if and only if one of the tame maps which sends the tuple of generators of $\GG(\Delta)$ to the tuple $\overline a$ defined by one of the finitely many coloured subgraphs of $\Gamma^e$ is an embedding. More precisely, we have the following lemma.

\begin{lemma}\label{lem:fmtameemb}
Given two simplicial graphs  $\Delta$ and $\Gamma$, $|V(\Delta)|=n$, $|V(\Gamma)|=m$, there is an algorithm that outputs at most $nm$ tuples $\overline g_i$ of $n$ elements from $\GG(\Gamma)$, $i=1, \dots, nm$ such that the following are equivalent:
\begin{itemize}
\item there is a tame embedding from $\GG(\Delta)$ to $\GG(\Gamma)$;
\item {\rm(}at least{\rm)} one of the maps $\varphi_i$, where $\varphi_i$ sends the $n$-tuple of generators of $\GG(\Delta)$ to the tuple $\overline g_i$, defines an embedding from $\GG(\Delta)$ to $\GG(\Gamma)$, $i=1, \dots, nm$.
\end{itemize}
\end{lemma}
\begin{proof}
Let $V(\Delta)= \{x_1, \dots, x_n\}$. Assume that there is a tame embedding $f$ from $\GG(\Delta)$ to $\GG(\Gamma)$. Since $f$ is a tame embedding, there are no repetitions of blocks in the block decompositions of $f(x_i)$, $i=1, \dots, n$, that is if a block appears in the decomposition of  $f(x_i)$ for some $i \in \{1, \dots, n\}$, then it does not appear in the block decomposition of $f(x_j)$ $i \ne j$. Hence, in the notation of Definition \ref{def:graphasstog}, we have that the graphs $\Lambda'$ and $\Lambda_{\overline g}$, where $\overline g =(f(x_1), \dots, f(x_n))$ are the same (as coloured graphs) and that $\Lambda_{\overline g} < \Gamma^e$ (and inflating $\Gamma$ is not needed). 

Let $W$ be the set of $i$-witnesses for the set $\mathcal W$ of words that define non-trivial elements in $\GG(\Delta)$. Since $f$ is an embedding, $\GG(\Gamma) \models \varphi_{\Delta, W}(\overline g)$. Then by Lemma \ref{lem:elemgraph}, we have that $\GG(\Lambda_{\overline g}) \models \varphi_{\Delta, W}(\overline a)$ and so $\Lambda_{\overline g} \in \bigcap \limits_{w\in W} wit(w)$. By Corollary \ref{cor:i-witness}, $\GG(\Lambda) \models \varphi_{\Delta, w}(\overline a)$, for all $w\in \mathcal W$. It follows that the map $X_i \to a_{i1} \cdots a_{ir_i}$ induces a tame embedding from $\GG(\Delta)$ to $\GG(\Lambda_{\overline g}) < \GG(\Gamma)$. 
\end{proof}

Furthermore, to decide if a tame homomorphism between pc groups is an embedding, it suffices to determine if it is injective in the set of $i$-witnesses. We state this result in the following lemma.

\begin{lemma}\label{cor:embfinitecond}
Let $\Delta$ and $\Gamma$ be simplicial graphs, $|V(\Delta)|=n$. Let $\mathcal W < F(x_1, \dots, x_n)$ be the set of words such that $w\ne 1$ in $\GG(\Delta)$ and let $W=\{w_1, \dots, w_s\}$ be the set of $i$-witnesses. Let $f:\GG(\Delta) \to \GG(\Gamma)$ be a tame map. Then if $f$ is injective on the set of $i$-witnesses $W$, then $f$ is an embedding.
\end{lemma}
\begin{proof}
Denote by $\overline g$ the tuple $(f(x_1), \dots, f(x_n))$. Since $f$ is a homomorphism and is injective on $W$, it follows that $\GG(\Gamma) \models \varphi_{\Delta, W}(\overline g)$. Then by Lemma \ref{lem:elemgraph}, we have that 
$$
\GG(\Lambda_{\overline g}) \models \varphi_{\Delta, W}(\overline a),
$$
and so $\Lambda_{\overline g} \in \bigcap\limits_{i=1, \dots,s} wit(w_i)$. By Corollary \ref{cor:i-witness}, we have that 
$$
\GG(\Lambda_{\overline g}) \models \varphi_{\Delta, w}(\overline a),
$$ 
for all $w\in \mathcal W$. It follows that $f': \GG(\Delta) \to \GG(\Lambda_{\overline g})$ induced by the map $f'(x_i)= a_{i1} \cdots a_{ir_i}$ is an embedding.

Since $f$ is a tame homomorphism, it follows that $\Lambda_{\overline g} < \Gamma^e$ (one does not need to inflate the graph $\Gamma$) and furthermore $a_{i1} \cdots a_{ir_i}=g_i$ (one does not need to pass to large enough powers of blocks to generate a pc group). Therefore, the tame homomorphism $f$ is an embedding.
\end{proof}

Summarising, we have shown that the existence of a tame embedding can be determined by checking if one of the finitely many maps (which can be described effectively) satisfies finitely many conditions, namely that the map sends commutators of generators that commute in $\GG(\Delta)$ to the identity (to assure that the map is a homomorphism) and that the finite set of $i$-witnesses is mapped non-trivially. More precisely, combining Lemmas \ref{lem:fmtameemb} and \ref{cor:embfinitecond}, we obtain

\begin{cor}\label{cor:chartameemb}
Let $\Delta$ and $\Gamma$ be simplicial graphs, $|V(\Delta)|=n$, $|V(\Gamma)|=m$. There exist $w_1, \dots, w_k \in \GG(\Delta)$ and $\overline{g_1}, \dots, \overline{g_{nm}} \in \GG(\Gamma)^n$ such that the following are equivalent:

\begin{itemize}
\item there exists a tame embedding from $\GG(\Delta)$ to $\GG(\Gamma)$;
\item at least one of the maps $\varphi_i$ that sends the tuple of generators $V(\Delta)$ to the tuple $\overline{g_i}$ is injective in the set of words $w_1, \dots, w_k$, $i=1, \dots, nm$, that is there exists $i\in \{1, \dots, nm\}$ such that $w_j=1$ in $\GG(\Delta)$ if and only if $\varphi_i(w_j)=1$ in $\GG(\Gamma)$ for all $j=1, \dots, k$.
\end{itemize}
\end{cor}

Therefore, in order to prove that the Tame Embedding Problem is decidable, it suffices to show that the maximal length of words in a (minimal) set of $i$-witnesses can be  bounded effectively by a number which depends on $\Delta$ and $\Gamma$, i.e. to construct an algorithm that given $\Delta$ and $\Gamma$ outputs a natural umber $N$ which bounds the length of the $i$-witnesses.

Although we do not present the details here, it is not difficult to see that the Embedding Problem between pc groups reduces to the Embedding Problem of the corresponding pc Lie algebras. In terms of Lie algebras, one can deduce from our results that there exists a number $N$ such that if a tame map $f$ between pc Lie algebras is injective on commutators of weight less than or equal to $N$, then it is an embedding. In this context we formulate the following 

\begin{question}
Does there exist an algorithm that given $\Delta$ and $\Gamma$ outputs a natural number $N$ so that if a tame map from the pc Lie algebra $\mathcal L (\Delta)$ to the pc Lie algebra $\mathcal L(\Gamma)$ is injective on commutators of weight less than or equal to $N$, then it is an embedding?
\end{question}

As we mentioned, a positive answer to the above question implies the decidability of the Tame Embedding Problem. We believe that this should be the case and so we formulate the following

\begin{conj}\label{conj:dectameemb}
The Tame Embedding Problem is decidable for pc groups, that is there is an algorithm that given two simplicial graphs $\Delta$ and $\Gamma$, decides whether or not there exists a tame embedding from $\GG(\Delta)$ to $\GG(\Gamma)$. 
\end{conj}

If the above conjecture holds, then it follows from Theorem \ref{cor:univequivstronemb} that there is an algorithm to decidable whether or not two pc groups are universally equivalent
\begin{cor} \label{cor:tembuni}
If the Tame Embedding Problem is decidable, then there is an algorithm to decide whether or not two given pc groups are universally equivalent.
\end{cor}

Although the general problem is open, there are classes of pc groups for which the Tame Embedding Problem is equivalent to the Extension Graph Embedding Problem and so both of these problems are decidable. We state this observation in the following corollaries.

\begin{cor}\label{cor:decunivth2dpc}
There is an algorithm that given an arbitrary simplicial graph $\Delta$ and a simplicial graph $\Gamma$ whose deflation graph is triangle-free decides whether or not $\GG(\Delta) \equiv_\forall \GG(\Gamma)$. In other words, if the  deflation graph of $\Gamma$ is triangle-free, then there is an algorithm to decide whether or not an arbitrary pc group is universally equivalent to $\GG(\Gamma)$.
\end{cor}
\begin{proof}
By Theorem \ref{cor:univequivstronemb}, $\GG(\Delta) \equiv_\forall \GG(\Gamma)$ if and only if there are tame embeddings $\GG(\Delta) < \GG(\Gamma_n)$ and $\GG(\Gamma) < \GG(\Delta_m)$. Since $\Gamma = \defl(\Gamma_n)$ is triangle-free, it follows from Theorem \ref{thm:defltfree} that there is a tame embedding of $\GG(\Delta)$ into $\GG(\Gamma_n)$ if and only if $\Delta < \Gamma_n^e$. Furthermore, since $\GG(\Delta) < \GG(\Gamma_n)$, then by Lemma \ref{lem:universalprop}, the deflation graph of $\Delta$ is triangle-free. Again from Theorem \ref{thm:defltfree}, it follows that $\GG(\Gamma) < \GG(\Delta_m)$ if and only if $\Gamma < \Delta_m^e$. Therefore, combining the above equivalences, we see that $\GG(\Delta) \equiv_\forall \GG(\Gamma)$ if and only if $\Delta < \Gamma_n^e$ and $\Gamma < \Delta_m^e$, which is decidable by Corollary \ref{cor:decEGE}.
\end{proof}

\begin{cor}
There is an algorithm that given and arbitrary simplicial graph $\Delta$ and a triangle-built graph $\Gamma$ decides whether or not $\GG(\Delta) \equiv_\forall \GG(\Gamma)$. In other words, if $\Gamma$ is triangle-built, then there is an algorithm to decide whether or not an arbitrary pc group is universally equivalent to $\GG(\Gamma)$.
\end{cor}
\begin{proof}
It suffices to notice that the inflation graph $\Gamma_n$ of a triangle-built graph is again triangle-built. Then by Theorem \ref{cor:univequivstronemb}, $\GG(\Delta) \equiv_\forall \GG(\Gamma)$ if and only if there are tame embeddings $\GG(\Delta) < \GG(\Gamma_n)$ and $\GG(\Gamma) < \GG(\Delta_m)$. By Theorem \ref{thm:CDKtbuilt}, the later is equivalent to $\Delta < \Gamma_n^e$ and $\Gamma < \Delta_m^e$, which is decidable by Corollary \ref{cor:decEGE}.
\end{proof}

\medskip

Our characterisation of the universal equivalence in algebraic terms allows us to transfer algebraic results on embeddability to model-theoretic ones. The following statements are immediate consequences of Theorem \ref{cor:univequivstronemb}.

In \cite{BKS}, Bestvina, Kleiner and Sageev introduce the class of atomic graphs (that is connected graphs with no valence 1 vertices, no cycles of length less than 5 and no separating closed stars of vertices) and establish quasi-isometric rigidity for pc groups defined by this class of graphs: two pc groups defined by atomic graphs are quasi-isometric if and only if they are isomorphic. We show that a similar type of rigidity holds for the elementary theory of the groups defined by atomic graphs, namely:

\begin{cor}\label{cor:atomic}\
Let $A_1$ and $A_2$ be atomic graphs. Then the following are equivalent:
\begin{itemize}
\item $\GG(A_1)$ and $\GG(A_2)$ are elementary equivalent;
\item $\GG(A_1)$ and $\GG(A_2)$ are universally equivalent;
\item $\GG(A_1)$ and $\GG(A_2)$ are isomorphic.
\end{itemize}
\end{cor}
\begin{proof}
It follows from Theorem \ref{cor:univequivstronemb} and Theorem \ref{thm:defltfree}, that $\GG(A_1)$ and $\GG(A_2)$ are universally equivalent if and only if $A_1 < {(A_2)}_n^e$ and $A_2 < {(A_1)}_m^e$. It is shown in \cite{C2} that the latter happens if and only the graphs $A_1$ and $A_2$ are isomorphic.
\end{proof}

\medskip 

\begin{cor}\label{cor:closedgraphproducts}
Universal equivalence in the class of pc groups is preserved under free products, direct products and more generally graph products:
$$
\GG(\Delta_i) \equiv_\forall \GG(\Gamma_i) \ i=1, \dots, |V(\Lambda)|  \Rightarrow \mathcal G(\Lambda, S(\Delta))  \equiv_\forall \mathcal G(\Lambda, S(\Gamma)),
$$
where $\Lambda$ is a simplicial graph, $S(\Delta)=\{\GG(\Delta_1), \dots, \GG(\Delta_{|V(\Lambda)|})\}$, $S(\Gamma)=\{\GG(\Gamma_1), \dots, \GG(\Gamma_{|V(\Lambda)|})\}$ and $\mathcal G(\Lambda,S )$ denotes the graph product with underlying graph $\Lambda$ and vertex groups $S$.
\end{cor}

\medskip

\begin{cor}\label{cor:join}
If $\Gamma$ is not a join and has more than one vertex, then
$$
\Th_\forall(\GG(\Gamma)) = \Th_\forall ( \mathop{\ast}\limits_{i=1}^{n} \GG(\Gamma)).
$$
\end{cor}

\medskip

\begin{cor}\label{cor:closedclasses}
The following classes of partially commutative groups are closed under universal equivalence, that is if $\GG(\Gamma)$ belongs to a class $\mathcal A$ described below and $\GG(\Delta)$ is universally equivalent to $\GG(\Gamma)$, then $\GG(\Delta)$ also belongs to $\mathcal A$:
\begin{itemize}
\item the class of pc groups whose deflation graph is a tree: $\{ \GG(\Gamma) \mid \defl(\Gamma) \hbox{ is a tree} \}$;
\item the class of pc groups whose deflation graph is triangle-free: $\{ \GG(\Gamma) \mid \defl(\Gamma) \hbox{ is triangle-free} \}$;
\item the class of weakly chordal pc groups: $\{ \GG(\Gamma) \mid \Gamma \hbox{ is $P_3$-free and $C_4$-free} \}$.
\end{itemize}
\end{cor}

\end{document}